\documentclass[12pt]{amsart}
\usepackage{amsfonts,amsmath,amssymb,amscd,mathrsfs}
\usepackage{xcolor}
\usepackage{soul}
\sethlcolor{cyan}
\usepackage{verbatim}
\usepackage{enumitem}
\usepackage{hyperref}
\usepackage{tikz-cd}
\newtheorem{theorem}{Theorem}[section]
\newtheorem{lemma}[theorem]{Lemma}
\newtheorem{corollary}[theorem]{Corollary}
\newtheorem{proposition}[theorem]{Proposition}

\theoremstyle{definition}
\newtheorem{definition}[theorem]{Definition}
\newtheorem{example}[theorem]{Example}

\theoremstyle{remark}
\newtheorem{remark}[theorem]{Remark}

\def \D{\mathbb{D}}
\def \N{\mathbb{N}}
\def \Z{\mathbb{Z}}

\def \C{\mathbb{C}}
\def \T{\mathbb{T}}

\newcommand{\clb}{\mathcal{B}}

\newcommand{\cle}{\mathcal{E}}
\newcommand{\clh}{\mathcal{H}}

\newcommand{\clm}{\mathcal{M}}
\newcommand{\cln}{\mathcal{N}}

\newcommand{\clk}{\mathcal{K}}

\newcommand{\clf}{\mathcal{F}}

\newcommand{\clo}{\mathcal{O}}

\newcommand{\ran}{\mathrm{ran}}

\newcommand{\rank}{\mathrm{rank}}

\newcommand{\ind}{\operatorname{ind}}

\newcommand{\ol}{\overline}
\newcommand\restr[2]{\ensuremath{\left.#1\right|_{#2}}}

\numberwithin{equation}{section}


\begin{document}

\title[Odometer maps on Fock spaces]{Odometer maps on Fock spaces: block decompositions, Toeplitz-type realizations, and the adjoint}
\author{Mansi Anil Suryawanshi}
\address{Department of Mathematics, Technion—Israel Institute of Technology, Haifa 32000, Israel.}
\email{suryawanshi@campus.technion.ac.il; mansisuryawanshi1@gmail.com}

\date{\today}

\maketitle	

\begin{abstract}
We study odometer maps $W_L$ on vector-valued full Fock spaces arising from Fock representations of the odometer semigroup. We obtain a canonical upper triangular block decomposition
\[
W_L=
\begin{pmatrix}
W_{11} & W_{12}\\
0 & W_{22}
\end{pmatrix},
\]
where $W_{11}$ is unitary and $W_{22}$ admits a Hardy space realization as an analytic Toeplitz operator $M_\Theta$. The associated symbol $\Theta\in H^\infty_{\mathcal{B}(\mathcal{E})}(\mathbb{D})$ is used to characterize the isometric, unitary, and invertible cases, as well as norm identities and Douglas-type factorization properties of $W_L$.

We also derive an explicit formula for $W_L^*$ for arbitrary bounded symbols $L$. In the isometric case, this identifies $\ker W_L^*$ with $\mathcal{E}_L\ominus L\mathcal{E}$, and hence $\operatorname{mult}(W_L)=\dim(\mathcal{E}_L\ominus L\mathcal{E})=\operatorname{mult}(M_\Theta)$. In the same setting, the condition $\dim(\mathcal{E}_L\ominus L\mathcal{E})<\infty$ is equivalent both to Fredholmness and to essential normality of $W_L$, with $\operatorname{ind}(W_L)=-\dim(\mathcal{E}_L\ominus L\mathcal{E})$. We further obtain Coburn-type spectral consequences and a necessary condition for hyponormality.
\end{abstract}

\section{Introduction}

The odometer semigroup arises naturally at the interface of combinatorial semigroup theory, topological dynamics, and noncommutative operator algebras. For an integer \(n\geq 1\), the \emph{odometer semigroup} \(\clo_n\) is the semigroup with generators
\(w,v_1,\ldots,v_n\) satisfying
\[
w v_i =
\begin{cases}
v_{i+1}, & 1 \leq i \leq n-1,\\
v_1 w, & i=n.
\end{cases}
\]
It is also known as the adding-machine semigroup. This semigroup is closely related to the positive Baumslag--Solitar monoid $BS(1,n)^+$; see, for instance, \cite{Foreman}. It also admits descriptions in terms of Zappa--Sz\'ep products \cite{Geb}.

The representation theory of semigroups and their associated
$C^*$-algebras forms an important theme in operator theory, with
connections to the work of Cuntz \cite{cuntz1977}, Nica \cite{Anica}, and Pimsner \cite{pimsner1997}. In the case of the odometer semigroup, related structures also appear in the study of topological subshifts, boundary quotient constructions, and scale-invariant dynamical systems \cite{brownlowe2012,julien2016,laca1996,li2022}. The dynamical origin of the subject is the classical topological odometer. Given integers $k_i\geq 2$, the odometer acts on the compact product space $\prod_{i\in\N}\mathbb Z_{k_i}$
by addition of the sequence $(1,0,0,\ldots)$ according to the usual carry rule \cite{Foreman}; see also \cite{Brin,Bruin,Del}. Odometer-type dynamical systems and their variants have appeared in the study of Toeplitz algebras \cite{Clark}, von Neumann algebras \cite{Fima}, $C^*$-algebras \cite{Spiel}, and semigroup \(C^*\)-algebras \cite{Brown,Hli,xli,Anica}.

In \cite{mansi2025}, the representation theory of the odometer semigroup was developed in the setting of vector-valued full Fock spaces. More precisely, Fock representations of $\clo_n$ were classified in terms of operator-valued symbols. For each $L\in \clb(\cle,\clf_n^2\otimes\cle)$, an associated odometer map $W_L\in \clb(\clf_n^2\otimes\cle)$ was introduced. It was shown that, if $S^\cle$ denotes the standard row
isometry of left creation operators on $\clf_n^2\otimes\cle$, then a pair $(W,S^\cle)$ is a Fock representation of $\clo_n$ if and only if there exists a uniquely determined symbol
$L\in \clb(\cle,\clf_n^2\otimes\cle)$ such that
\[
W=W_L.
\]
Thus the classification of Fock representations is reduced to the study of the corresponding operator $W_L$. While \cite{mansi2025} established the existence of these maps and characterized the conditions under which they give rise to isometric, unitary, and Nica-covariant representations, the intrinsic geometric and operator-theoretic structure of $W_L$ remained to be investigated.

The construction of $W_L$ is reminiscent of the symbol calculus for noncommutative Toeplitz operators developed by Popescu \cite{popescu}. Nevertheless, unlike the Toeplitz setting, the symbol in the odometer framework
is allowed to vary throughout the entire Fock space, leading to a substantially
richer structure. Motivated by the concluding remarks of \cite{mansi2025}, the purpose of the present paper is to undertake a systematic study of the geometry, function-theoretic model, and operator-theoretic properties of odometer maps.

The central point of the paper is that an odometer map, although defined on
a noncommutative Fock space and governed by a noncommutative carrying rule,
contains a canonical single-variable analytic Toeplitz component. This
allows one to transfer questions about isometry, invertibility, norm,
Fredholmness, and defect multiplicity to the associated analytic symbol.

The first main result of the paper is a canonical block decomposition of
\(W_L\), together with a Hardy space realization of its Toeplitz-type block. In Section~\ref{sec:BDHSR}, we introduce two canonical closed subspaces
\(\clm, \cln \subseteq \clf_n^2\otimes\cle\), which yield the orthogonal
decompositions
\[
\clf_n^2\otimes\cle=\cln\oplus\cln^\perp
\qquad\text{and}\qquad
\clf_n^2\otimes\cle=\clm\oplus\clm^\perp.
\]
With respect to these decompositions, \(W_L\) admits the upper triangular block
representation
\[
W_L=
\begin{pmatrix}
W_{11} & W_{12}\\
0 & W_{22}
\end{pmatrix}.
\]

We observe that $W_{11}$ is unitary. Consequently, whenever
$\cle\neq\{0\}$ and $n\geq 2$, the odometer map $W_L$ cannot be compact. For $n=1$, we show that \emph{the only compact odometer maps are the zero maps}. This mirrors the corresponding compactness phenomenon for Toeplitz operators. The off-diagonal block $W_{12}:\cln^\perp\to \clm$  is either zero or has infinite rank. We refer to $W_{22}:\cln^\perp\to \clm^\perp$ as the \emph{Toeplitz-type operator} associated with $W_L$, and we define its \emph{Toeplitz-type realization} on a vector-valued Hardy space. A key result of the paper shows that this realization is unitarily equivalent to the analytic Toeplitz operator
\[
M_\Theta \quad \text{on } H^2_{\cle}(\D),
\]
where the symbol $\Theta\in H^\infty_{\clb(\cle)}(\D)$
is uniquely determined by \(L\). We also present a collection of explicit examples illustrating how different choices of the symbol \(L\) give rise to concrete
analytic symbols \(\Theta\), including constant symbols, shift-type symbols, and Blaschke-product symbols. 

For \(n=1\), the odometer map is essentially an analytic Toeplitz-type
operator. In contrast, for \(n\geq 2\), the canonical block decomposition   
measures the deviation of \(W_L\) from the classical Toeplitz framework. In the special  case \(L\cle\subseteq \clm^\perp\), 
\(W_L\) becomes the direct sum of a unitary operator and a Toeplitz-type
operator. Moreover, the Toeplitz-type realization of the principal block translates
questions about the multivariable noncommutative Fock space operator \(W_L\)
into questions about the single-variable analytic symbol \(\Theta\). This brings
Hardy space operator theory and \(H^\infty\)-function methods into the analysis
of odometer representations and establishes a direct connection with the
classical framework of analytic Toeplitz operators and Sz.-Nagy--Foiaş dilation
theory \cite{SzNagyFoias,Paulsen2002,Pisier2001}.

In Section~\ref{sec:appl_block_decomp}, we apply this decomposition to study
operator-theoretic properties of \(W_L\). First, we prove a
Douglas-type factorization theorem for odometer maps; see
Subsection~\ref{sec:Douglas-factorization-odometer-maps}. By bridging the noncommutative geometry of the free Fock space with the classical Hardy algebra, we unlock a novel, purely operator-theoretic mechanism for such factorizations. More precisely, we
show that factorization at the level of symbols is equivalent to a
corresponding factorization of the associated odometer maps. Namely, if \(L_1,L_2\in\clb(\cle,\clf_n^2\otimes\cle)\), then
\(\ran L_1\subseteq \ran L_2\) if and only if there exists a bounded operator \(\Gamma\in\clb(\clf_n^2\otimes\cle)\) satisfying
\[
\Gamma(\Omega\otimes\cle)\subseteq \Omega\otimes\cle
\quad\text{and}\quad W_{L_1}=W_{L_2}\Gamma.
\]
Equivalently, the following diagram commutes:
\[
\begin{tikzcd}
& \clf_n^2\otimes \cle
\arrow[d, "\Gamma"]
\arrow[dl, swap, "W_{L_1}"]
\\
\clf_n^2\otimes \cle
&
\clf_n^2\otimes \cle
\arrow[l, "W_{L_2}"]
\end{tikzcd}
\]
Under the geometric assumptions $L_i\cle \subseteq \clm^{\perp},\, i=1,2$, the multivariable Douglas factorization of the odometer maps descends into a single-variable analytic factorization of their associated analytic symbols. If
\(\Theta_i\in H^\infty_{\clb(\cle)}(\D)\) are the analytic symbols associated with \(L_i\), then \(\ran L_1\subseteq \ran L_2\) implies the existence of $\Theta_3 \in H^\infty_{\clb(\cle)}(\D)$ such that
\[
\Theta_1(z) = \Theta_2(z) \Theta_3 (z), \qquad z \in \D, 
\]
where $\Theta_3$ is a constant operator-valued function given by $\Theta_3(z) \equiv C$. Here $C \in \clb(\cle)$ satisfies $L_1 =L_2C.$ Thus the Fock-space factorization of odometer maps induces a single-variable analytic multiplier factorization of the corresponding Hardy space symbols.

Second, the Hardy space realization yields characterizations of several
basic operator-theoretic properties of $W_L$. In
Subsection~\ref{sec:Characterizations}, we give new characterizations of the isometry and unitarity of $W_L$ in terms of the associated analytic symbol $\Theta$, yielding shorter and more conceptual proofs of the corresponding criteria from \cite{mansi2025}. We also obtain a new invertibility criterion: $W_L$ is invertible precisely when $\Theta$ is invertible in $H^\infty_{\clb(\cle)}(\D)$. Furthermore, we illustrate by examples the distinction between unitarity and invertibility in terms of the analytic symbol. 

In Subsection~\ref{sec:norm_W_L}, we obtain norm formulae. For an arbitrary symbol $L\in\clb(\cle,\clf_n^2\otimes\cle)$, it was shown in \cite[Remark~2.4]{mansi2025} that $\|L\|\leq \|W_L\|\leq \|L\|+1.$ Under the natural assumption
$L\cle\subseteq \clm^\perp$, we obtain the exact norm
\[
\|W_L\|=\max\{1,\|\Theta\|_\infty\},
\]
where $\Theta$ is the analytic symbol associated with the Toeplitz-type block of $W_L$. In particular, if \(\|\Theta\|_\infty\leq 1\), then \(\|W_L\|=1\) and hence the standard Sz.-Nagy--Foias dilation-theoretic consequences apply, including the existence of a unitary dilation and a contractive functional calculus for the disk algebra. Under a stronger geometric hypothesis, the above formula further reduces to
\[
\|W_L\|=\max\{1,\|L\|\}.
\]

A second main theme of the paper is the adjoint of an odometer map, detailed in Section~\ref{sec:The adjoint of odometer map}. Due to the highly noncommutative and symbol-dependent action of the odometer map, explicitly resolving its adjoint on the full Fock space requires a careful
analysis of the underlying odometer combinatorics. While \cite{mansi2025} extracted an adjoint formula strictly under the rigid assumption that $W_L$ is an isometry, we overcome these hurdles to derive a closed-form expression for $W_L^*$ for \emph{every} bounded symbol $L$. This formula provides a useful technical tool for the subsequent analysis. As a consequence, we recover the previously known isometric adjoint formula as a special case, thereby substantially extending the earlier theory. 

Section~\ref{sec:applications_adjoint} is devoted to consequences of the adjoint formula. The main object in this part is the space $\cle_L$ defined as 
\begin{equation*} 
\cle_{L} = \ol{\text{span}}\{e_{1}^{ \otimes m}\otimes\eta : m\in\Z_{+},\eta\in\cle\} \ominus \ol{\text{span}}\{e_{1}^{\otimes p}\otimes L\zeta : p\geq 1,\zeta\in\cle\}.
\end{equation*}
We show that the position of $L\cle$ inside $\cle_L$
controls the isometric and Fredholm structure of the odometer map. For an isometry $W_L$, we prove the wandering subspace for the shift part in its Wold decomposition is identified with the defect space $\cle_L\ominus L\cle$:
\[
\ker W_L^*=\cle_L\ominus L\cle.
\]  
The following example shows that for an isometric odometer map, $\dim (\cle_L\ominus L\cle) $ is precisely the multiplicity of the shift part of $W_L$. Moreover, this defect space is reflected in the associated analytic Toeplitz operator $M_\Theta$. In particular, the multiplicity of the shift part of $W_L$ is encoded by the associated analytic symbol:
\[
\ker M_\Theta^*=U_1(\cle_L\ominus L\cle),
\qquad
\operatorname{mult}(W_L)=\operatorname{mult}(M_\Theta).
\]

For an isometric odometer map $W_L$, the case $\dim(\cle_L\ominus L\cle)=0$ characterizes unitary maps via the Wold decomposition, and gives a shorter proof of \cite[Theorem~4.4]{mansi2025}. The condition $\dim(\cle_L\ominus L\cle)<\infty$
is equivalent both to the Fredholmness of \(W_L\) and to the essential normality of \(W_L\). In this case,
\[
\operatorname{ind}(W_L)=-\,\dim(\cle_L\ominus L\cle).
\]
Under the same finite-defect hypothesis, we prove a Coburn-type spectral theorem:
For every $\lambda\in\D$, the operator $W_L-\lambda I$ is Fredholm and
\[
\ind(W_L-\lambda I)
=
-\dim(\cle_L\ominus L\cle).
\]
Furthermore, in this case, if $W_L$ is a proper isometry, then its essential spectrum is the unit circle. We also illustrate, through examples, that the dimension of the defect space coincides with the absolute value of the corresponding Fredholm index.
Finally, the adjoint formula yields a necessary condition for
hyponormality: if $W_L$ is hyponormal, then
\[
\|L\eta\|\geq \|\eta\|,\qquad \eta\in\cle.
\]
We illustrate through an example that this condition is not sufficient.

The paper is organized as follows. Section~\ref{sec:Preliminaries} contains the necessary preliminaries on full Fock spaces, odometer maps,
and vector-valued Hardy spaces. Section~\ref{sec:BDHSR} establishes the canonical block decomposition of the odometer map and the Hardy space realization of its Toeplitz-type block. Section~\ref{sec:appl_block_decomp} develops applications of the block decomposition, including Douglas-type factorization, characterizations of isometric, unitary, and invertible
odometer maps, and norm formulae. Section~\ref{sec:Examples} contains examples illustrating how the analytic symbol $\Theta$ can be computed explicitly from the defining
symbol $L$. These examples also illustrate the characterization criteria for $W_L$ to be an isometry, a unitary operator, or an invertible operator. In Section~\ref{sec:The adjoint of odometer map}, we compute explicitly the adjoint of an arbitrary bounded odometer map. Section~\ref{sec:applications_adjoint} uses this formula to analyze
\(\ker W_L^*\), the Wold decomposition, Fredholmness, essential normality,
Coburn-type spectral phenomena, and hyponormality.

\section{Preliminaries} \label{sec:Preliminaries}

In this section, we fix notation and collect standard results that will be used throughout the paper. For background on operator theory, Hardy spaces, and $C^*$-algebras, we refer the reader to \cite{arveson1976,douglas1998banach,SzNagyFoias}. Let $\clh$ be a Hilbert space. Throughout this paper, all Hilbert spaces are assumed to be separable and over $\mathbb C$. If $\clm$ is a closed subspace of $\clh$, we write $P_{\clm}$ for the orthogonal projection of $\clh$ onto $\clm$. An operator $N\in\clb(\clh)$ is called \emph{normal} if
\[
N^*N=NN^*.
\]
An operator $V\in\clb(\clh)$ is called an \emph{isometry} if
\[
V^*V=I.
\]
An isometry $V$ is called \emph{pure} if $V^{*m}\xrightarrow[m\to\infty]{\mathrm{SOT}}0.$
An isometry $U\in\clb(\clh)$ is called \emph{unitary} if it is
surjective; equivalently,
\[
U^*U=UU^*=I.
\]

We now recall the classical vector-valued Hardy space model together with the basic operator-theoretic notions that will be used repeatedly in the sequel. For a Hilbert space $\cle$, let
\[
H^2_{\cle}(\D)
=
\left\{
f(z)=\sum_{m=0}^{\infty} z^m \eta_m :
\eta_m\in \cle,\;
\sum_{m=0}^{\infty}\|\eta_m\|^2<\infty
\right\}
\]
denote the \(\cle\)-valued Hardy space on the unit disc \(\D\).
Equipped with the norm $\|f\|^2 = \sum_{m=0}^{\infty}\|\eta_m\|^2,$
it is a Hilbert space. The unilateral shift on \(H^2_{\cle}(\D)\) is the multiplication operator
\[
M_z f(z)=z f(z),
\qquad f\in H^2_{\cle}(\D).
\]
We define
\begin{align*}
 H^\infty_{\clb(\cle)}(\D) &=
\left\{ \Theta:\D\to \clb(\cle)\,:\, \Theta \text{ is analytic and }
\sup_{z\in\D}\|\Theta(z)\|<\infty
\right\}. \\
\end{align*}
For an operator-valued bounded analytic function $\Theta \in H^\infty_{\clb(\cle)}(\D),$
the corresponding analytic Toeplitz operator is defined by
\[
M_\Theta f(z)=\Theta(z)f(z),
\qquad f\in H^2_{\cle}(\D).
\]
It is well known that $\|M_\Theta\|=\|\Theta\|_\infty
:=
\sup_{|z|<1}\|\Theta(z)\|.$ Moreover, a bounded operator $T\in \clb(H^2_{\cle}(\D))$ is an analytic Toeplitz operator if and only if \[ M_zT=TM_z. \] Equivalently, \[ \{M_z\}' = \{M_\Theta:\Theta\in H^\infty_{\clb(\cle)}(\D)\}. \]
The multiplication operator $M_\Theta:H^2_{\cle}(\D)\to H^2_{\cle}(\D)$ is an isometry if and only if $\Theta$ is \emph{inner}, that is,
\[
\Theta(\zeta)^*\Theta(\zeta)=I_\cle
\quad \text{for a.e. } \zeta\in \partial\D.
\]
Moreover, $M_\Theta$ is unitary if and only if $\Theta$ is a constant unitary operator-valued function.

A noncommutative analogue of the Hardy space is the full Fock space, which plays a central role in the theory of Cuntz--Pimsner algebras and semigroup crossed products \cite{davidson1996, Anica, pimsner1997}. The full Fock space over $\C^n$ is denoted by
\[
\clf_n^2=\C\Omega\oplus\bigoplus_{k=1}^{\infty}(\C^n)^{\otimes k},
\]
where $\Omega$ is the vacuum vector. For a word
$\mu=\mu_1\cdots\mu_k\in F_n^+$, we write $e_\mu=e_{\mu_1}\otimes\cdots\otimes e_{\mu_k},$
and $e_\emptyset=\Omega$. Then $\{e_\mu:\mu\in F_n^+\}$
is the canonical orthonormal basis of \(\clf_n^2\). The length of a word $\mu$ is denoted by $|\mu|$. For $k\in\{1,\ldots,n\}$ and $m\in \Z_+$, we abbreviate $e_k^{\otimes m}$ by $e_k^m$.
For each $i=1,\ldots,n$, the $i$-th creation operator $S_i$ on $\clf_n^2$ is defined by
\[
S_i f=e_i\otimes f.
\]
More generally, for a Hilbert space $\cle$, we write $S^\cle=(S_1\otimes I_\cle,\ldots,S_n\otimes I_\cle)$
for the amplified left creation tuple on $\clf_n^2\otimes\cle$. This is a pure row isometry; that is, for every  $f\in\clf_n^2\otimes\cle$,
\[
\lim_{m\to\infty}
\sum_{\substack{|\mu|=m\\ \mu\in F_n^+}}
\|(S_\mu^*\otimes I_\cle)f\|^2=0,
\]
where $S_\mu=S_{\mu_1}\cdots S_{\mu_m}$ whenever $ \mu=\mu_1\cdots\mu_m\in F_n^+$ \cite[Remark 1.1]{popescu1989}. These creation operators play a fundamental role in noncommutative operator theory and arise naturally in free analytic models \cite{Frazho,popescu18}.

Our main focus is on odometer maps associated with Fock representations of the odometer semigroup. For a natural number $n\ge 1$, the odometer semigroup $\clo_n$ is generated by $\{w, v_1,\ldots,v_n\}$, subject to the relations
\[
wv_i=
\begin{cases}
v_{i+1}, & 1\le i\le n-1,\\
v_1w, & i=n.
\end{cases}
\]

A representation of \(\clo_n\) on \(\clh\) consists of a pair \((W,T)\), where $T=(T_1,\ldots,T_n)$
is a row operator on $\clh$ and $W\in \clb(\clh),$
such that
\[
WT_i=
\begin{cases}
T_{i+1}, & i=1,\ldots,n-1,\\[4pt]
T_1W, & i=n.
\end{cases}
\]

When $\clh=\clf_n^2\otimes \cle$ and $(W,S^\cle)$ is a representation of $\clo_n$, we call $(W,S^\cle)$ a \emph{Fock representation}. If, in addition, $W$ is an isometry (respectively, a unitary or a contraction), then $(W,S^\cle)$ is called an \emph{isometric} (respectively, \emph{unitary} or \emph{contractive}) Fock representation.

It was shown in \cite[Theorem 2.2]{mansi2025} that $(W,S^{\cle})$ is a Fock representation if and only if there exists $L \in \clb(\cle,\clf_n^2 \otimes \cle)$
such that
\[
W=W_L,
\]
where $W_L \in \clb(\clf_n^2 \otimes \cle)$ is given by
\[
W_L(\Omega \otimes \eta)=L\eta,
\qquad \eta \in \cle,
\]
and for $e_\mu=e_{\mu_1}\otimes \cdots \otimes e_{\mu_m}\in  F_n^+,\, \mu\neq \emptyset,$
define
\begin{equation}\label{eq:defn-W_L}
 W_L(e_\mu\otimes \eta)=
\begin{cases}
e_{\mu_1+1}\otimes e_{\mu_2}\otimes \cdots \otimes e_{\mu_m}\otimes \eta,
& \mu_1\neq n,\\[6pt]
e_1\otimes e_{\mu_2+1}\otimes e_{\mu_3}\otimes \cdots \otimes e_{\mu_m}\otimes \eta,
& \mu_1=n,\ \mu_2\neq n,\\[6pt]
e_1^{ 2}\otimes e_{\mu_3+1}\otimes e_{\mu_4}\otimes \cdots \otimes e_{\mu_m}\otimes \eta,
& \mu_1=\mu_2=n,\ \mu_3\neq n,\\[6pt]
\vdots\\[4pt]
e_1^{ m}\otimes L\eta,
& \mu_1=\cdots=\mu_m=n.
\end{cases}   
\end{equation}
Moreover, the symbol $L$ is uniquely determined, and $\|L\|\le \|W_L\|\le \|L\|+1.$
The operator $W_L$ is called the \emph{odometer map} with symbol $L$. A notable feature of this construction is that the role of symbols is analogous to that in the theory of noncommutative Toeplitz operators; see \cite{popescu} and \cite[Remark 2.4]{mansi2025}. 

The structure of $W_L$ was investigated under additional assumptions such as isometry, unitarity, and Nica-covariance in \cite{mansi2025}.
Let $(W, S^\cle)$ be an isometric Fock representation of $\clo_n$. We say that $(W, S^\cle)$ is a Nica-covariant representation of $\clo_n$ if 
$$W^*(S_1 \otimes I_\cle) = (S_n \otimes I_\cle)W^*.$$
For a symbol $L$, we define the subspace $\cle_L \subseteq \clf_n^2 \otimes \cle$ as
\begin{equation} \label{eq:EL}
\cle_{L} = \ol{\text{span}}\{e_{1}^{ m}\otimes\eta : m\in\Z_{+},\eta\in\cle\} \ominus \ol{\text{span}}\{e_{1}^{ p}\otimes L\zeta : p\geq 1,\zeta\in\cle\}.
\end{equation}
The following structural characterizations were obtained in \cite[Theorem 4.4]{mansi2025}. For convenience, we recall the relevant result below.
\begin{theorem}\label{thm:fock_classification}
Let $\cle$ be a Hilbert space, and let $L \in \clb(\cle, \clf_n^2 \otimes \cle)$. Then
\begin{enumerate}
    \item $(W_L, S^\cle)$ is an isometric representation if and only if $L$ is an isometry and $L\cle \subseteq \cle_L$.
    \item An isometric Fock representation \((W_L,S^\cle)\) is unitary if and only if
\(L\cle=\Omega\otimes\cle\).
    \item For an isometric symbol $L$, the map $W_L$ is unitary if and only if $L\cle = \Omega \otimes \cle$.
\end{enumerate}
\end{theorem}

A central role in the sequel will be played by an upper triangular block decomposition of the odometer map $W_L$.
We shall use compactness in the analysis of the off-diagonal block. An operator $K\in\clb(\clh)$ is compact if, whenever $x_n\to 0$ weakly in $\clh$, one has $\|Kx_n\|\to 0.$
The ideal of compact operators on $\clh$ is denoted by $\clk(\clh)$. We then apply the preceding block decomposition of $W_L$ to investigate factorization properties of odometer maps using Douglas' range inclusion theorem (cf. \cite{Douglas1966}). 
\begin{theorem}[Douglas' lemma]\label{thm:douglas_lemma}
Let \(A\in \clb(\clh_1,\clh_3)\) and
\(B\in \clb(\clh_2,\clh_3)\). Then the following assertions are equivalent:
\begin{enumerate}
    \item \(\ran A\subseteq \ran B\);
    \item there exists \(C\in \clb(\clh_1,\clh_2)\) such that
    \[
    A=BC;
    \]
    \item there exists \(\lambda>0\) such that
    \[
    AA^*\leq \lambda BB^*.
    \]
\end{enumerate}
\end{theorem}

We also compute the adjoint of the odometer map. 
If \(A\in\clb(\clh_1,\clh_2)\), then \(A^*\in
\clb(\clh_2,\clh_1)\) is determined by
\[
\langle Ax,y\rangle_{\clh_2}
=
\langle x,A^*y\rangle_{\clh_1},
\qquad x\in\clh_1,\ y\in\clh_2.
\]
We then use the explicit formula for $W_L^*$ to derive several
operator-theoretic consequences of the odometer map. The first consequence is a description of the kernel of $W_L^*$. Recall that, for a bounded operator $T\in\clb(\clh)$, we have $(\ran T)^\perp=\ker T^*.$

\begin{theorem}[von Neumann--Wold decomposition]
Let $V\in \clb(\clh)$ be an isometry. Then there exist orthogonal reducing subspaces $\clh=\clh_s\oplus \clh_u$,
where
\[
\clh_s=\bigoplus_{k=0}^{\infty}V^k(\ker V^*)
\quad\text{and}\quad
\clh_u=\bigcap_{k=0}^{\infty}V^k\clh,
\]
such that
\begin{enumerate}
    \item $\restr{V}{\clh_s}$ is a pure isometry;
    \item $\restr{V}{\clh_u}$ is unitary.
\end{enumerate}
\end{theorem}
The space \(\ker V^*\) is the wandering subspace for the shift part of \(V\).
We define the multiplicity of the unilateral shift part in the Wold
decomposition of \(V\) by
\[
\operatorname{mult}(V):=\dim\ker V^*.
\]

We shall also use the following standard Fredholm terminology
(cf. \cite{atkinson1948}). An operator \(T\in\clb(\clh)\) is called
\emph{upper semi-Fredholm} if $\dim\ker T<\infty$ and $\ran T \text{ is closed}.$
It is called \emph{lower semi-Fredholm} if $\dim\ker T^*<\infty$ and $\ran T \text{ is closed}.$
The operator \(T\) is called \emph{semi-Fredholm} if it is either upper
semi-Fredholm or lower semi-Fredholm.

An operator \(T\in\clb(\clh)\) is called \emph{Fredholm} if
\[
\dim\ker T<\infty,\qquad
\ran T \text{ is closed},
\qquad
\dim\ker T^*<\infty.
\]
In this case, the Fredholm index of \(T\) is defined by
\[
\ind(T)=\dim\ker T-\dim\ker T^*.
\]
We shall use the fact that the Fredholm index is locally constant on the
semi-Fredholm set.

We shall also recall the notion of essential normality and the essential spectrum. An operator $T\in\clb(\clh)$ is said to be essentially normal if its self-commutator is compact, that is,
\[
T^*T-TT^*\in \clk(\clh).
\]
For an operator $T\in\clb(\clh)$, the essential spectrum is defined by
\[
\sigma_{\mathrm{ess}}(T)
=
\{\lambda\in\mathbb C: T-\lambda I \text{ is not Fredholm}\}.
\]
Equivalently, $\lambda\notin\sigma_{\mathrm{ess}}(T)$ precisely when $T-\lambda I$ is Fredholm. We shall combine this with the local constancy of the Fredholm index on the semi-Fredholm set to obtain a Coburn-type spectral result for odometer maps.

Finally, we shall use the following elementary characterization of hyponormality. An operator $T\in\clb(\clh)$ is called \emph{hyponormal} if $T^*T\geq TT^*,$ or equivalently,
\[
\|T^*x\|\leq \|Tx\|,
\qquad x\in\clh.
\]
An operator $T\in\clb(\clh)$ is called \emph{subnormal} if there exist a Hilbert space $\clk\supseteq\clh$ and a normal operator $N\in\clb(\clk)$ such that $\clh$ is invariant for $N$ and
\[
T=\restr{N}{\clh}.
\]
Every subnormal operator is hyponormal. Therefore, any obstruction to hyponormality is also an obstruction to subnormality.

\section{Block decomposition and Hardy space realization}\label{sec:BDHSR}

Throughout the paper, $\cle$ denotes a Hilbert space.  For $L \in \clb(\cle, \clf_n^2 \otimes \cle)$, let $W_L \in \clb( \clf_n^2 \otimes \cle)$ denote the corresponding odometer map as defined by \eqref{eq:defn-W_L}. Let $\{e_\mu : \mu \in F_n^+\}$ be the canonical orthonormal basis of $\clf_n^2$. 
Let 
\[
\begin{aligned}
\clm_0
&=
\{\mu\in F_n^+:\ |\mu|\geq 1
\text{ and } \mu_i\neq 1 \text{ for some } i\},\\
\cln_0
&=
\{\mu\in F_n^+:\ |\mu|\geq 1
\text{ and } \mu_i\neq n \text{ for some } i\}.
\end{aligned}
\]
We define
\begin{equation}\label{eq:Clm-Cln}
\clm= \ol{\rm span}\{ e_{\mu} \otimes \xi :\mu \in \clm_0,\, \xi \in \cle  \}, \quad
\cln=\ol{\rm span}\{ e_{\mu} \otimes \xi : \mu \in \cln_0,\, \xi \in \cle\}.
\end{equation}
It follows immediately that
\begin{equation}\label{eq:Clm-Cln-Perp}
\clm^{\perp}=\ol{\rm span}\{ e_{1}^{m} \otimes \xi: m \geq 0,\,  \xi \in \cle \},\quad \cln^{\perp}
=\ol{\rm span}\{ e_{n}^{m} \otimes \xi:  m \geq 0,\, \xi \in \cle \},
\end{equation}
and hence we obtain the orthogonal decompositions
\[
\clf_n^2\otimes\cle=\clm\oplus\clm^\perp,
\qquad
\clf_n^2\otimes\cle=\cln\oplus\cln^\perp.
\]
Let \(\gamma=\gamma_1\cdots\gamma_m\in\clm_0\). Let \(k\) be the smallest index such that \(\gamma_k\neq 1\), and set
\[
e_{\gamma-1}
=
e_n^{k-1}\otimes e_{\gamma_k-1}\otimes e_{\gamma_{k+1}}\otimes\cdots\otimes e_{\gamma_m}.
\]
For \(\mu=\mu_1\cdots\mu_m\in\cln_0\), let \(k\) be the smallest index such that \(\mu_k\neq n\), and set
\begin{equation}\label{eq: e_mu+1}
e_{\mu+1}
=
e_1^{k-1}\otimes e_{\mu_k+1}\otimes e_{\mu_{k+1}}\otimes \cdots \otimes e_{\mu_m}.
\end{equation}
Equivalently, for every \(\xi\in\cle\),
$W_L(e_\mu\otimes \xi)=e_{\mu+1}\otimes \xi.$
Note that the odometer map satisfies $W_L(\cln)=\clm.$
Let
\[
W_{11}:=\restr{W_L}{\cln}, \quad W_{12}:= \restr{P_{\clm} W_L}{\cln^{\perp}}, \quad W_{22}:= \restr{P_{\clm^{\perp}} W_L}{\cln^{\perp}}.
\]
Then, with respect to the orthogonal decompositions $\clf_n^2 \otimes \cle = \cln \oplus \cln^\perp$ and $\clf_n^2 \otimes \cle = \clm \oplus \clm^\perp$, the operator $W_L : \cln \oplus \cln^\perp \to \clm \oplus \clm^\perp$
admits the following block matrix representation:
\begin{equation}\label{eq:matrix representation}
 W_L=
\begin{pmatrix}
W_{11} & W_{12} \\
0 & W_{22}
\end{pmatrix}.   
\end{equation}
This block decomposition of $W_L$ provides a convenient framework for the analysis of its geometric and operator-theoretic properties.
Much of the analysis in the subsequent sections will rely on this decomposition and the interaction between the blocks $W_{11}$, $W_{12}$, and $W_{22}$.

Observe that the restriction \(W_{11}:\cln \to \clm\) is unitary. Since
\(\cln\) and \(\clm\) are infinite-dimensional whenever \(n\geq 2\) and
\(\cle\neq \{0\}\), it follows that \(W_L\) cannot be compact for \(n\geq 2\).
Thus, the only remaining case in which compactness can occur is \(n=1\).
The following proposition gives equivalent criteria for compactness in this case.

\begin{proposition}\label{prop:WL_compact}
Assume that \(\cle \neq \{0\}\) and \(n=1\). Then an odometer map
\(W_L \in \clb(\clf_1^2 \otimes \cle)\) is compact if and only if
\(W_L=0\).
\end{proposition}
\begin{proof}
If \(W_L=0\), then \(W_L\) is trivially compact.
Conversely, assume that \(W_L\) is compact. Since \(n=1\), we have
\(\cln=\clm=\{0\}\).
Thus $\clf_1^2\otimes\cle=\cln^\perp.$
Suppose \(L\neq0\). Then there exists
\(\eta\in\cle\) such that \(L\eta\neq0\). For \(p\ge0\), set
\[
y_p=e_1^p\otimes\eta\in \clf_1^2\otimes\cle.
\]
The vectors \(\{y_p\}_{p\ge0}\) are mutually orthogonal, and hence
\(y_p\to0\) weakly as \(p\to\infty\). Since \(W_L\) is compact, it follows
that $\|W_Ly_p\|\to0.$
On the other hand, by the definition of the odometer map for \(n=1\), $W_Ly_p
=
(S_1^p\otimes I_\cle)L\eta.$
Since \(S_1^p\otimes I_\cle\) is an isometry, we obtain
\[
\|W_Ly_p\|
=
\|L\eta\|
>0,
\qquad p\ge0.
\]
This contradicts \(\|W_Ly_p\|\to0\). 
Therefore $L=0$, and hence $W_L=0.$
\end{proof}

Since $W_{11}$ is unitary, the non-unitary features of \(W_L\) are governed by the remaining blocks
\(W_{12}\) and \(W_{22}\).
\begin{lemma} \label{lem:W12W22}
Let $W_L\in \clb(\clf_n^2\otimes \cle)$ be an odometer map with block decomposition \eqref{eq:matrix representation}. Then
\[
W_{12}=\restr{W_{(P_{\clm}L)}}{\cln^\perp},
\qquad
W_{22}=\restr{W_{(P_{\clm^\perp}L)}}{\cln^\perp}.
\]
\end{lemma}
\begin{proof}
It suffices to verify the identities on the dense spanning set $\{e_n^p\otimes \eta : p\ge 0,\ \eta\in\cle\}$
of $\cln^\perp$. Let $p\ge 0$ and $\eta\in\cle$. By definition of the odometer map, $W_L(e_n^p\otimes \eta)=e_1^p\otimes L\eta.$
Hence
\[
P_{\clm}W_L(e_n^p\otimes \eta)
=
P_{\clm}(e_1^p\otimes L\eta)
=
e_1^p\otimes P_{\clm}L\eta
=
W_{(P_{\clm}L)}(e_n^p\otimes \eta).
\]
Therefore, $W_{12}=\restr{W_{(P_{\clm}L)}}{\cln^\perp}.$
A similar computation yields $W_{22}=\restr{W_{(P_{\clm^\perp}L)}}{\cln^\perp}.$
\end{proof}

We next investigate the compactness properties of the off-diagonal block $W_{12}$. The following proposition shows that compactness of $W_{12}$ is highly restrictive and, in fact, forces $W_{12}$ to vanish identically.

\begin{proposition}\label{prop:W12_compactness}
Consider the off-diagonal block $W_{12}:\cln^\perp \to \clm$ appearing in the block decomposition \eqref{eq:matrix representation}. Then
\[
W_{12} \text{ is compact}
\quad \Longleftrightarrow \quad
W_{12}=0.
\]
\end{proposition}
\begin{proof}
If $n=1,$ then $W_{12}=0$. Therefore let $n \geq 2$. Assume that $W_{12}$ is compact. Fix $\eta\in\cle$ and consider the sequence $y_p:=e_n^p\otimes \eta\in \cln^\perp,\, p\ge 0.$
The vectors $\{y_p\}_{p\ge 0}$ are orthogonal, and hence $y_p\to 0$ weakly as $p\to\infty$. Compactness of $W_{12}$ implies $\|W_{12}y_p\|\to 0.$ Notice that
\[
W_{12}y_p
=
W_{(P_{\clm}L)}(e_n^p\otimes \eta)
=
e_1^p\otimes P_{\clm}L\eta
=
(S_1^p\otimes I_\cle)P_{\clm}L\eta.
\]
Since $S_1^p\otimes I_\cle$ is an isometry for each $p$, we obtain
\[
\lim_{p \to \infty} \|W_{12} y_p\| =\|P_{\clm}L\eta\|= 0.
\]
As $\eta$ was arbitrary, $P_{\clm}L=0$, that is, $L\cle \subseteq \clm^\perp$.
Then clearly  $W_{12}=\restr{W_{(P_{\clm} L)}}{\cln^\perp} =0.$ The converse direction is immediate, completing the proof.
\end{proof}

\begin{corollary}
The operator $W_{12}$ is either zero or has infinite rank.
\end{corollary}
\begin{proof}
If $W_{12}$ has finite rank, then it is compact. By Proposition~\ref{prop:W12_compactness}, this forces $W_{12}=0.$
Therefore, if $W_{12}\neq 0$, it must have infinite rank.
\end{proof}

\begin{remark}
If $L\cle \subseteq \clm^\perp$ (for example, when $W_{12}$ is compact or when $W_L$ is an isometry), then $W_{12}=0$. Consequently, the decomposition \eqref{eq:matrix representation} reduces to the block-diagonal form $W_L = W_{11} \oplus W_{22}.$    
\end{remark}

We now turn to the analysis of the second diagonal block $W_{22}$. To this end, consider the canonical unitaries $U_1:\clm^\perp \to H^2_{\cle}(\D),$ and $U_n:\cln^\perp \to H^2_{\cle}(\D),$
defined by
\[
U_1(e_1^{p}\otimes \eta)=z^p\eta
=
U_n(e_n^{p}\otimes \eta),
\qquad p\ge 0,\ \eta\in\cle.
\]

\begin{theorem}\label{prop:hardy_space_multiplier}
For \(r\geq 0\), let
\(L_r\in\clb(\cle)\) be defined by $L_r:=(\langle e_1^r,\cdot\rangle\otimes I_\cle)L.$
Then
\[
U_1W_{22}U_n^*=M_\Theta
\quad \text{on } H^2_{\cle}(\D),
\]
where \(M_\Theta\) is the analytic Toeplitz operator with uniquely determined
symbol \(\Theta\in H^\infty_{\clb(\cle)}(\D)\) given by
\[
\Theta(z)=\sum_{r\geq 0}z^rL_r.
\]
\end{theorem}
\begin{proof}
For $p\ge 0$ and $\eta\in\cle$, 
\[
U_1W_{22}U_n^*(z^p\eta)
=
U_1W_{22}(e_n^p\otimes \eta)
=
U_1(e_1^p\otimes P_{\clm^\perp}L\eta).
\]
Therefore,
\[
M_z(U_1W_{22}U_n^*)(z^p\eta)
=
U_1(e_1^{p+1}\otimes P_{\clm^\perp}L\eta)
=
U_1W_{22}(e_n^{p+1}\otimes \eta)
=
(U_1W_{22}U_n^*)M_z(z^p\eta).
\]
Since vector-valued polynomials are dense in $H^2_{\cle}(\D)$, it follows that $U_1W_{22}U_n^*\in \{M_z\}'.$
Therefore, there exists a unique $\Theta\in H^\infty_{\clb(\cle)}(\D)$
such that $$U_1W_{22}U_n^*=M_\Theta.$$

To identify the symbol $\Theta$, fix $\eta\in\cle$ and consider the constant function $f\in H^2_{\cle}(\D),$ defined by $f(z)\equiv \eta.$
As $L_r\eta$ is the coefficient of $e_1^r$ in the Fourier expansion of $L\eta$, we can write
\begin{equation} \label{eq:Ltilde in terms of Lr}
 P_{\clm^{\perp}} L\eta
=
\sum_{r\ge 0} e_1^r\otimes L_r\eta,
\qquad \eta\in\cle.  
\end{equation}   
Therefore,
\[
\Theta(z)\eta
=
(M_\Theta f)(z)
=
(U_1W_{22}U_n^*f)(z)
=
U_1(P_{\clm^{\perp}}L\eta)
=
\sum_{r\ge 0} z^rL_r\eta.
\]
Hence $\Theta(z)=\sum_{r\ge 0} z^rL_r.$
\end{proof}

\begin{definition}\label{rem:toeplitz_block}
In view of Theorem~\ref{prop:hardy_space_multiplier}, we call $W_{22}:\cln^\perp\to \clm^\perp$
the \emph{Toeplitz-type operator} associated with \(W_L\), or simply the
\emph{Toeplitz-type operator} when \(W_L\) is clear from the context.

Henceforth, we
identify $U_1W_{22}U_n^*$ with the multiplication operator \(M_\Theta\) and call
\(\Theta\in H^\infty_{\clb(\cle)}(\D)\) the \emph{analytic symbol} associated
with \(W_L\).

Let \(V:\clm^\perp\to \cln^\perp\) be the canonical unitary defined by $V(e_1^p\otimes \eta)=e_n^p\otimes \eta,
\, p\geq 0,\ \eta\in\cle.$ Then
\[
W_{22}V=U_1^*M_\Theta U_1.
\]
The operator \(W_{22}V\in \clb(\clm^\perp)\) is called the \emph{Toeplitz-type realization} of \(W_L\).
\end{definition}

Combining the preceding results, we arrive at the first fundamental result: the structural decomposition of the odometer map.

\begin{theorem}\label{thm:WL_structural_sum}
With respect to the orthogonal decompositions $\clf_n^2\otimes \cle=\cln\oplus \cln^\perp$ and $\clf_n^2\otimes \cle=\clm\oplus \clm^\perp,$
the odometer map $W_L$ admits the block decomposition
\[
W_L=
\begin{pmatrix}
W_{11} & W_{12}\\
0 & W_{22}
\end{pmatrix},
\]
where
\begin{enumerate}
    \item $W_{11}:\cln\to \clm$ is unitary;

        \item $W_{12}:\cln^\perp\to \clm$ is either zero or has infinite rank;

    \item \(W_{22}:\cln^\perp\to \clm^\perp\) is the Toeplitz-type operator
associated with \(W_L\), and its Toeplitz-type realization is unitarily
equivalent to the analytic Toeplitz operator \(M_\Theta\) on
\(H^2_{\cle}(\D)\), 
where \(\Theta\in H^\infty_{\clb(\cle)}(\D)\) is uniquely determined by \(L\):
\[
\Theta(z)=\sum_{r\geq 0}z^rL_r,
\qquad
L_r=(\langle e_1^r,\cdot\rangle\otimes I_\cle)L.
\]
\end{enumerate}
\end{theorem}
In Section~\ref{sec:Examples}, we present several examples in which the associated analytic symbol is computed explicitly from a given symbol \(L\), using Theorems~\ref{prop:hardy_space_multiplier} and~\ref{thm:WL_structural_sum}.

\section{Applications of the Block Decomposition} \label{sec:appl_block_decomp}
In this section, we apply the block decomposition from
Theorem~\ref{thm:WL_structural_sum} to investigate the operator-theoretic
properties of the odometer map \(W_L\).
\subsection{Factorization of odometer maps}\label{sec:Douglas-factorization-odometer-maps}
The block matrix structure of the odometer map enables a precise transfer of factorization properties from the level of defining symbols to the full Fock space. In particular, the following result applies Douglas’ lemma \cite{Douglas1966} to characterize when one odometer map factors through another in terms of their defining symbols. 
First, assume that $\ran L_1\subseteq \ran L_2$. By Douglas' lemma
\cite{Douglas1966}, there exists an operator $C\in \clb(\cle)$ such that
$$L_1=L_2C.$$
We fix such an operator \(C\) throughout this section, and 
show that this factorization at the level of symbols is
equivalent to a corresponding factorization of the odometer maps.

\begin{theorem}\label{thm:douglas_odometer}
Let $L_1, L_2 \in \clb(\cle, \clf_n^2 \otimes \cle)$. Then $\ran L_1 \subseteq \ran L_2$ if and only if there exists a bounded operator $\Gamma \in \clb(\clf_n^2 \otimes \cle)$ 
satisfying
\[
\Gamma(\Omega \otimes \cle)\subseteq \Omega \otimes \cle, \text{ and } W_{L_1}=W_{L_2}\Gamma.
\]
Further, if $W_{L_2}$ is an isometry, then $W_{L_1}$ is an isometry if and only if $\Gamma$ is an isometry.
\end{theorem}

\begin{proof}
Let $\ran L_1\subseteq \ran L_2$. 
Define $\Gamma_C:\cln^\perp \to \cln^\perp$ by
\[
\Gamma_C(e_n^{p}\otimes \eta) = e_n^{p}\otimes C\eta, \qquad p\ge 0,\ \eta\in\cle.
\]
Since $\cln^\perp=\bigoplus_{p\ge0}(e_n^{ p}\otimes\cle)$, $\Gamma_C$ is a bounded operator with $\|\Gamma_C\|=\|C\|$. With respect to the orthogonal decomposition $\clf_n^2\otimes\cle=\cln\oplus \cln^\perp,$ we define
\[
\Gamma=
\begin{pmatrix}
I_{\cln} & 0\\
0 & \Gamma_C
\end{pmatrix}.
\]
By construction, $\Gamma$ is bounded with norm $\|\Gamma\| = \max\{1, \|C\|\}$. For any $\eta \in \cle,\, \Gamma(\Omega\otimes\eta)=\Omega\otimes C\eta\in \Omega\otimes\cle$ implies $$\Gamma(\Omega\otimes\cle) \subseteq \Omega\otimes\cle.$$ 

We now verify the factorization $W_{L_1}=W_{L_2}\Gamma$. For any $x\in \cln$, identifying $x$ with $x\oplus 0$, we have $\Gamma x = x$. Because the action of any odometer map on the space $\cln$ is completely independent of the symbol, we obtain
\[
W_{L_1}x=W_{L_2}x=W_{L_2}\Gamma x, \qquad x\in \cln.
\]
Now let $e_n^{p}\otimes\eta\in \cln^\perp$. Then
\[
W_{L_2}\Gamma(e_n^{p}\otimes\eta)
=
W_{L_2}(e_n^{p}\otimes C\eta)
=
e_1^{ p}\otimes L_2C\eta
=
e_1^{ p}\otimes L_1\eta
=
W_{L_1}(e_n^{p}\otimes\eta).
\]
Since the operators agree on both $\cln$ and $\cln^\perp$, we conclude that $W_{L_1}=W_{L_2}\Gamma$.

Conversely, let $\Gamma \in \clb(\clf_n^2 \otimes \cle)$ be such that $\Gamma(\Omega\otimes\cle)\subseteq \Omega\otimes\cle$ and   $W_{L_1}=W_{L_2}\Gamma$. Because the closed subspace $\Omega\otimes\cle$ is naturally isomorphic to $\cle$, the restriction of $\Gamma$ to this subspace defines a unique bounded operator $C\in \clb(\cle)$ such that
\[
\Gamma(\Omega\otimes\eta)=\Omega\otimes C\eta, \qquad \eta\in\cle.
\]
Now
\[
L_1\eta
=
W_{L_1}(\Omega\otimes\eta)
=
W_{L_2}\Gamma(\Omega\otimes\eta)
=
W_{L_2}(\Omega\otimes C\eta)
=
L_2C\eta.
\]
Therefore, $L_1=L_2C$, and hence $\ran L_1\subseteq \ran L_2$.

Finally, assume that $W_{L_2}$ is an isometry. Since $W_{L_1}=W_{L_2}\Gamma$, we have
\[
W_{L_1}^*W_{L_1}
=
\Gamma^*W_{L_2}^*W_{L_2}\Gamma
=
\Gamma^*\Gamma.
\]
Therefore, $W_{L_1}^*W_{L_1} = I$ if and only if $\Gamma^*\Gamma = I$.
\end{proof}

The preceding factorization theorem also has a natural Hardy space
interpretation. In the theory of vector-valued Hardy spaces, factorization problems for operator-valued analytic functions, particularly Schur-class and inner
functions, form a highly nontrivial part of the subject. Such factorization problems are closely related to Leech's theorem
\cite{Leech2014} and to range inclusion problems for Toeplitz and Hankel operators on vector-valued Hardy spaces \cite{Bhuia}. They also arise naturally in the Sz.-Nagy--Foiaş dilation-theoretic framework; see \cite[Chapter~5]{SzNagyFoias}.

Under the geometric assumptions \(L_i\cle\subseteq\clm^\perp\), the
Douglas-type factorization of odometer maps provides a bridge between
Fock-space operator theory and classical Hardy space multiplier factorization.
The following corollary shows that range inclusion for Fock-space symbols
gives rise to analytic factorization of the corresponding Hardy space
multipliers.

\begin{corollary}\label{cor:douglas_toeplitz_factorization}
Let $L_i \in \clb(\cle, \clf_n^2 \otimes \cle)$ for $i=1,2$ be bounded symbols satisfying $L_i\cle \subseteq \clm^{\perp}$, and let $\Theta_i \in H^\infty_{\clb(\cle)}(\D)$ be their associated analytic symbols. 
If $\ran L_1 \subseteq \ran L_2$, then there exists $\Theta_3 \in H^\infty_{\clb(\cle)}(\D)$ such that
\[
\Theta_1(z) = \Theta_2(z) \Theta_3 (z), \qquad z \in \D, 
\]
where $\Theta_3$ is a constant operator-valued function given by $\Theta_3(z) \equiv C$.
\end{corollary}

\begin{proof}
Let $W_{L_i}$ admit the block decomposition:
\[
W_{L_i} =
\begin{pmatrix}
W^{(i)}_{11} & W^{(i)}_{12} \\
0 & W^{(i)}_{22}
\end{pmatrix}, \quad i = 1,2.
\]
Since $L_i\cle \subseteq \clm^\perp$, we have $W_{12}^{(1)}=W_{12}^{(2)}=0$. Moreover, the block $W_{11}^{(i)}$ is independent of the symbol. Hence, setting $W_{11}:=W_{11}^{(1)}=W_{11}^{(2)},$ we obtain $W_{L_i}=W_{11}\oplus W_{22}^{(i)}$ for $i=1,2.$
By Theorem~\ref{thm:douglas_odometer}, $W_{L_1}=W_{L_2}\Gamma$, where $\Gamma=I_{\cln}\oplus \Gamma_C$. Therefore, 
\[
W_{11}\oplus W_{22}^{(1)}
=
(W_{11}\oplus W_{22}^{(2)})(I_{\cln}\oplus\Gamma_C)
=
W_{11}\oplus W_{22}^{(2)}\Gamma_C,
\]
and hence, $W_{22}^{(1)}=W_{22}^{(2)}\Gamma_C.$
By Theorem~\ref{prop:hardy_space_multiplier}, $M_{\Theta_i}=U_1W_{22}^{(i)}U_n^*$ for $i=1,2.$
Hence, 
\[
M_{\Theta_1} = M_{\Theta_2}(U_n\Gamma_CU_n^*).
\]
Consider $U_n \Gamma_C U_n^* \in \clb(H^2_\cle(\D))$. For $p\ge 0$ and $\eta \in \cle$, 
\[
U_n \Gamma_C U_n^*(z^p \eta) 
= U_n \Gamma_C (e_n^{p} \otimes \eta) 
= U_n (e_n^{p} \otimes C\eta) 
= z^p C\eta.
\]
Therefore, $U_n \Gamma_C U_n^* = I_{H^2(\D)} \otimes C$. Notice that
\begin{align*}
(I_{H^2(\D)} \otimes C) M_z (z^p \eta) = z^{p+1} C\eta=
M_z (I_{H^2(\D)} \otimes C)(z^p \eta).
\end{align*}
Since the operators coincide on a dense subset, it follows that $(I_{H^2(\D)} \otimes C) M_z = M_z (I_{H^2(\D)} \otimes C)$. Hence there exists a unique symbol $\Theta_3 \in H^\infty_{\clb(\cle)}(\D)$ such that $I_{H^2(\D)} \otimes C = M_{\Theta_3}$. Therefore $M_{\Theta_1} = M_{\Theta_2} M_{\Theta_3},$
completing the proof.
\end{proof}

\begin{remark}
The conclusion of Corollary~\ref{cor:douglas_toeplitz_factorization} 
also follows directly.
We get
\[
L_{1,r}
=
(\langle e_1^r,\cdot\rangle\otimes I_\cle)L_1
=
(\langle e_1^r,\cdot\rangle\otimes I_\cle)L_2C
=
L_{2,r}C,
\qquad r\geq0.
\]
Therefore, by Theorem~\ref{prop:hardy_space_multiplier},
\[
\Theta_1(z)
=
\sum_{r\geq0}z^rL_{1,r}
=
\sum_{r\geq0}z^rL_{2,r}C
=
\Theta_2(z)C,
\qquad z\in\D.
\]
\end{remark}

\subsection{Characterizations} \label{sec:Characterizations}
Now, using the block decomposition \eqref{eq:matrix representation}, we obtain unified characterizations of odometer maps that are isometric,
unitary, or invertible in terms of the associated analytic symbol,
thereby refining \cite[Theorem~3.1]{mansi2025}. 

\begin{theorem}\label{thm:operator_characterization_symbol}
Let $L \in \clb(\cle,\clf_n^2\otimes \cle)$
and let $\Theta \in H^\infty_{\clb(\cle)}(\D)$
be the analytic symbol associated with  $W_L.$
Then 
\begin{enumerate}
    \item \(W_L\) is an isometry if and only if \(L(\cle)\subseteq \clm^\perp\), and \(\Theta\) is inner.

   \item \(W_L\) is unitary if and only if $L\cle\subseteq \clm^\perp$
and \(\Theta\) is a constant unitary operator-valued function.

   \item \(W_L\) is invertible if and only if \(\Theta\) is invertible in
\(H^\infty_{\clb(\cle)}(\D)\). 
 \end{enumerate}
\end{theorem}

\begin{proof}
\begin{enumerate}
        \item 
Suppose that $W_L$ is an isometry. Pick $\eta, \gamma \in \cle$, and let $e_\mu \otimes \gamma \in \clm$. By the definition of the odometer map, there exists a unique preimage $e_{\mu-1} \otimes  \gamma \in \cln$ such that $W_L( e_{\mu-1} \otimes  \gamma)=e_\mu \otimes \gamma$.
Since $W_L$ is an isometry, we have
\[
\langle L\eta,\, e_\mu \otimes \gamma \rangle
=
\langle W_L(\Omega \otimes \eta),\, W_L( e_{\mu-1} \otimes  \gamma) \rangle
= \langle \Omega \otimes \eta,\, e_{\mu-1} \otimes  \gamma \rangle = 0,
\]
where the final equality holds because $e_{\mu-1}\otimes\gamma \in \cln$, and hence is orthogonal to the vacuum space. Hence, $L\cle \subseteq \clm^\perp$, and $W_{12} = 0$. Therefore, $W_L = W_{11} \oplus W_{22}$. Since $W_{11}$ is unitary, it follows that $W_{22}$ is an isometry. As $U_1$ and $U_n$ are unitaries, Theorem~\ref{prop:hardy_space_multiplier} implies
\[
M_{\Theta}^*M_{\Theta}
=
U_n W_{22}^* W_{22} U_n^*
=
I_{H^2_\cle(\D)},
\]
which shows that $M_{\Theta}$ is an isometry. Equivalently, $\Theta$ is inner.

Conversely, if $L(\cle)\subseteq \clm^\perp$ then $W_{12}=0.$ Since $\Theta$ is inner, the Toeplitz operator $M_\Theta$ is an isometry. Therefore, $W_{22}=U_1^*M_\Theta U_n$ is an isometry. Because $W_{11}$ is unitary, the diagonal block operator $W_L = W_{11} \oplus W_{22}$ is an isometry.

\item By (1) above \(W_L\) is unitary if
and only if $L\cle\subseteq\clm^\perp$ and \(M_\Theta\) is unitary, which holds if and only if \(\Theta\) is
a constant unitary operator-valued function.
\item
Since $M_\Theta = U_1 W_{22} U_n^*,$ we have $W_{22}\text{ is invertible}$ if and only if $
M_\Theta\text{ is invertible}.$
Equivalently, \(\Theta\) is invertible in
\(H^\infty_{\clb(\cle)}(\D)\), that is,
\(\Theta^{-1}\in H^\infty_{\clb(\cle)}(\D)\).
Assume first that $\Theta \in H^\infty_{\clb(\cle)}(\D)$ is invertible. Define $X:\clm\oplus \clm^\perp \to \cln\oplus \cln^\perp$ by
\[
X:=
\begin{pmatrix} 
W_{11}^* & -\,W_{11}^*W_{12}W_{22}^{-1}\\
0 & W_{22}^{-1}
\end{pmatrix}.
\]
Since $W_{11}$ is unitary, a direct computation gives
\[
W_LX=
\begin{pmatrix}
W_{11} & W_{12}\\
0 & W_{22}
\end{pmatrix}
\begin{pmatrix}
W_{11}^* & -\,W_{11}^*W_{12}W_{22}^{-1}\\
0 & W_{22}^{-1}
\end{pmatrix}
=
\begin{pmatrix}
I_{\clm} & 0\\
0 & I_{\clm^\perp}
\end{pmatrix}.
\]
That is, $W_LX=I_{\clm\oplus \clm^\perp}$. A similar computation implies $XW_L= I_{\cln\oplus \cln^\perp}.$
Thus $X=W_L^{-1}$, and hence $W_L$ is invertible.

Conversely, assume that $W_L$ is invertible. We prove that $W_{22}:\cln^{\perp} \to \clm^{\perp}$ is both surjective and injective.
First, let $y\in \clm^\perp$. Since $W_L$ is onto, there exists $u\oplus v\in \cln\oplus \cln^\perp$ such that $W_L(u\oplus v)=0\oplus y.$
Equivalently,
\[
W_{11}u+W_{12}v=0, \qquad W_{22}v=y.
\]
Hence $y\in \ran W_{22}$, so $W_{22}$ is surjective.
Next, let $v\in \cln^\perp$ satisfy $W_{22}v=0$. Since $W_{11}:\cln\to \clm$ is unitary, there exists $u\in \cln$ such that $W_{11}u=-\,W_{12}v.$
Then
\[
W_L(u\oplus v)=\big(W_{11}u+W_{12}v\big)\oplus W_{22}v=0\oplus 0.
\]
As $W_L$ is injective, it follows that $u=0$ and $v=0$. Hence $\ker W_{22}=\{0\}$. Since $W_{22}: \cln^\perp \to \clm^\perp$ is a bounded bijection between Hilbert spaces, the bounded inverse theorem implies that \(W_{22}^{-1}\in \clb(\clm^\perp,\cln^\perp)\).
Consequently $\Theta$ is invertible. 
\end{enumerate}
\end{proof}

The following example shows that invertibility of \(W_L\) does not force the
associated analytic symbol \(\Theta\) to be a constant unitary
operator-valued function, in contrast to the unitary case. 

\begin{example}
 Let \(n\geq2\),
and fix a scalar \(a\in\C\) with \(0<|a|<1\). Define
\(L:\cle\to \clf_n^2\otimes\cle\) by
\[
L\eta=\Omega\otimes\eta+a\,e_1\otimes\eta,
\qquad \eta\in\cle.
\]
The associated coefficient operators are
\[
L_0=I_\cle,\qquad L_1=aI_\cle,\qquad L_r=0\quad(r\geq2).
\]
Hence the associated analytic symbol is
\[
\Theta(z)=I_\cle+azI_\cle=(1+az)I_\cle,\qquad z\in\D.
\]
Since \(0<|a|<1\), the scalar function \(1+az\) has no zeros in \(\D\), and $\frac{1}{1+az}\in H^\infty(\D).$
Therefore
\[
\Theta^{-1}(z)=\frac{1}{1+az}I_\cle \in H^\infty_{\clb(\cle)}(\D).
\]
Thus \(\Theta\) is invertible in
\(H^\infty_{\clb(\cle)}(\D)\). By
Theorem~\ref{thm:operator_characterization_symbol}, \(W_L\) is invertible.

On the other hand, \(\Theta\) is not a constant unitary operator-valued
function, since \(a\neq0\). Hence \(W_L\) is not unitary. Indeed, for any nonzero
\(\eta\in\cle\),
\[
\|W_L(\Omega\otimes\eta)\|^2
=
\|L\eta\|^2
=
\|\Omega\otimes\eta+a\,e_1\otimes\eta\|^2
=
(1+|a|^2)\|\eta\|^2 > \|\eta\|^2.
\]
Thus \(W_L\) is not an isometry, and consequently \(W_L\) is not unitary.
\end{example}

\subsection{Norm of the Odometer Map }
\label{sec:norm_W_L}
We compute the norm of $W_L$. For an arbitrary bounded symbol $L$, it is known that $\|L\| \leq \|W_L\| \leq \|L\|+1$
(cf.~\cite[Remark~2.4]{mansi2025}). Although these inequalities provide useful general bounds, they do not yield an exact norm formula. We show that, under the additional assumption $L\cle \subseteq \clm^\perp,$
the preceding theorem yields an exact norm formula. Let $\|\Theta\|_\infty$ denote the norm of $\Theta$ in $H^\infty_{\clb(\cle)}(\D)$.

\begin{proposition}
\label{normWL}
Let $W_L \in \clb(\clf_n^2 \otimes \cle)$ be the odometer map with $L\cle \subseteq \clm^\perp$, and let $\Theta
\in H^\infty_{\clb(\cle)}(\D)$
be the associated analytic symbol. Then
\[
\|W_L\| = \max\{1,\|\Theta\|_{\infty}\}.
\]
\end{proposition}

\begin{proof}
Since $L\cle \subseteq \clm^\perp$, $W_{12}=0.$
Hence the block decomposition \eqref{eq:matrix representation} reduces to $W_L=
W_{11} \oplus W_{22}$,
and therefore
\[
\|W_L\|=\max\{\|W_{11}\|,\|W_{22}\|\}.
\]
Since \(M_\Theta=U_1W_{22}U_n^*\) and \(U_1,U_n\) are unitaries, we have
$\|W_{22}\|=\|M_\Theta\|=\|\Theta\|_{\infty}.$
Since $W_{11}$ is unitary, we obtain
\[
\|W_L\|
=
\max\{1,\|\Theta\|_{\infty}\}.
\]
\end{proof}
Notice that, if $L\cle \subseteq \clm^\perp$, then for every $\eta\in\cle$, $L\eta
=
W_L(\Omega\otimes \eta)
=
W_{22}(\Omega\otimes \eta).$
Consequently, $\|L\eta\|
\leq
\|W_{22}\|\,\|\eta\|,$
and hence Proposition~\ref{normWL} implies
\[
\|L\|
\leq \|W_{22}\|=
\|\Theta\|_\infty.
\]
Under the stronger geometric hypothesis $L\cle \subseteq \cle_L,$ we prove that $\|L\|=\|\Theta\|_\infty,$
and therefore the norm of $W_L$ is completely determined by the norm of $L$.

\begin{proposition} \label{prop:norm_theta_L}
Let \(L\in\clb(\cle,\clf_n^2\otimes\cle)\) satisfy \(L\cle\subseteq\cle_L\).
Then
$\|W_L\|=\max\{1,\|L\|\}.$
\end{proposition}

\begin{proof}
Let $\Theta \in H^\infty_{\clb(\cle)}(\D)$ denote the analytic symbol associated with $W_L$. Since $\cle_L \subseteq \clm^\perp$, Proposition~\ref{normWL} implies  $\|\Theta\|_\infty=\|W_{22}\|$. Thus, it suffices to prove that $\|W_{22}\| = \|L\|$. We already know $ \|L\| \leq \|W_{22}\|$. For the reverse inequality, let $f = \sum_{p=0}^N e_n^{ p}\otimes \eta_p \in \cln^\perp$ be a finite sum, where $\eta_p \in \cle$. Then $W_{22}f = \sum_{p=0}^N e_1^{ p}\otimes L\eta_p.$
As $L\cle \subseteq \cle_L$, we get $L\cle \perp (S_1^k \otimes I_\cle)L\cle$ for all $k \ge 1$. Since \(S_1\otimes I_\cle\) is an isometry, it follows that
\[
\|W_{22}f\|^2
=
\sum_{p=0}^N \|e_1^{p} \otimes L\eta_p\|^2
=
\sum_{p=0}^N \|L\eta_p\|^2
\le
\|L\|^2 \sum_{p=0}^N \|\eta_p\|^2
=
\|L\|^2 \|f\|^2.
\]
Since finite sums of this form are dense in $\cln^\perp$, this estimate extends by continuity to all of $\cln^\perp$, establishing that $\|W_{22}\| \le \|L\|$, and hence $\|W_{22}\| = \|L\|$.
\end{proof}
The following example shows that the assumption 
$L\cle \subseteq \cle_L$ in Proposition~\ref{prop:norm_theta_L}
cannot be weakened to the condition 
$L\cle \subseteq \clm^\perp$ alone in order to obtain $\|W_L\|=\max\{1,\|L\|\}.$

\begin{example}
Let $\cle = \mathbb{C}$, and define $L: \mathbb C \to \clf_n^2 $ by
\[
L(1)=\Omega + e_1.
\]
Then $\|L\|
=
\|\Omega + e_1\|
=
\sqrt{2}.$
The associated analytic symbol is
\[
\Theta(z)=1+z,
\]
and hence $\|\Theta\|_\infty
=
\sup_{|z|<1}|1+z|
=
2.$
Therefore, by Proposition~\ref{normWL},
$$\|W_L\|
=
\max\{1,\|\Theta\|_\infty\}
=
2
\neq
\max\{1,\|L\|\}
=
\sqrt{2}.$$
\end{example}

\begin{remark}
Let $W_L \in \clb(\clf_n^2\otimes\cle)$ be an odometer map satisfying
$L\cle\subseteq \clm^\perp$, and let $\Theta$ be its associated analytic
symbol. Then the spectral radius of $W_L$ satisfies $r(W_L)\le \max\{1,\|\Theta\|_\infty\}.$
In particular, if $\|\Theta\|_\infty\le 1$,  the exact norm evaluates to $\|W_L\| = 1$, making $W_L$ a contraction. Consequently, $W_L$ admits a unitary dilation in a larger Hilbert space (cf. \cite{SzNagyFoias,Paulsen2002,Pisier2001}). Moreover, by von Neumann's inequality,
$\|p(W_L)\| \le \sup_{|z|\le 1}|p(z)|$
for every polynomial $p$. Thus \(W_L\) admits a contractive functional calculus for the disk algebra \(A(\D)\). If, in addition, \(W_L\) is absolutely continuous in the sense of Sz.-Nagy--Foia\c{s}, then this calculus extends to the usual
\(H^\infty(\D)\)-functional calculus; namely, for every \(f\in H^\infty(\D)\),
\[
f(W_L)=\operatorname{SOT}\!-\!\lim_{r\to 1^-} f_r(W_L),
\qquad f_r(z)=f(rz).
\]
\end{remark}

\section{Examples}\label{sec:Examples}
In this section, we present a collection of examples illustrating the
structural decomposition in Theorem~\ref{thm:WL_structural_sum} and the
characterization results in Theorem~\ref{thm:operator_characterization_symbol}.
These examples show how different choices of the symbol \(L\) give rise to
explicit analytic symbols \(\Theta\) associated with \(W_L\). Recall that, for $r \geq 0$, \(L_r\in\clb(\cle)\) is defined by $L_r:=(\langle e_1^r,\cdot\rangle\otimes I_\cle)L.$

\begin{example} \label{ex:toeplitz symbol is M_z}
Let
$L:\cle\to \clf_n^2\otimes \cle$ be defined by
\[
L(\eta)=e_1\otimes \eta,
\qquad \eta\in \cle.
\]
Then 
$L_1 = I_\cle,$ and $L_r = 0, \, \forall r \neq 1.$ 
Therefore,
\[
\Theta(z)
=
\sum_{r=0}^{\infty} z^r L_r
=
z I_{\cle},
\qquad
\text{and }
\qquad
M_{\Theta}=M_z
\quad \text{on } H^2_{\cle}(\D).
\]
We verify directly that $U_1 W_{22} U_n^*=M_{z}.$ For $p \geq 0$ and $\eta \in \cle$, consider $e_n^{p} \otimes \eta \in \cln^\perp$. Then 
\[
U_1 W_{22} U_n^*(z^p \eta) = U_1 W_{22} (e_n^{p} \otimes \eta) = U_1 (e_1^{p} \otimes e_1 \otimes \eta) = z^{p+1} \eta=M_z (z^p \eta).
\]
Since the linear span of \(\{z^p\eta:p\geq0,\ \eta\in\cle\}\) is dense in $H^2_\cle(\D)$,
we get the required equality. Since \(L\cle\subseteq\clm^\perp\) and \(\Theta(z)=zI_\cle\) is inner,
Theorem~\ref{thm:operator_characterization_symbol} implies that \(W_L\) is an
isometry.
The constant symbol arises in the special case: If $L(\eta)=\Omega\otimes\eta$, then $\Theta(z)=I_\cle.$
\end{example}

\begin{example}
Let $\cle$ be an infinite-dimensional separable Hilbert space with orthonormal basis
$\{h_p\}_{p=0}^\infty$.
Define $L:\cle\to \clf_n^2\otimes \cle$ by
\[
L(h_p)=e_1^{ p}\otimes h_p,
\quad p\ge 0.
\]
For each $r\ge 0$, 
\[
L_r h_p
=
(\langle e_1^{r},\cdot\rangle\otimes I_{\cle})
(e_1^{p}\otimes h_p)
=
\delta_{r,p}h_p.
\]
Therefore, $ L_r=P_r,$
where $P_r$ is the orthogonal projection onto $\C h_r$, and $$\Theta(z)=\sum_{r=0}^\infty z^rP_r.$$
Moreover, for each $p\ge 0$, $\Theta(z)h_p=z^p h_p.$
Hence
\[
M_\Theta(f(z)h_p)=z^p f(z)h_p,
\qquad f\in H^2(\D).
\]
With respect to the orthogonal decomposition $H^2_{\cle}(\D)=\displaystyle\bigoplus_{p=0}^\infty H^2(\D)\otimes h_p,$
we obtain
\[
M_\Theta=\bigoplus_{p=0}^\infty M_{z^p}.
\]
A straightforward computation shows $U_1 W_{22} U_n^*=M_{\Theta}.$ 
Theorem~\ref{thm:fock_classification} (1) guarantees $W_L$ is an isometry. Consequently, $\Theta$ must be inner by Theorem~\ref{thm:operator_characterization_symbol}. Indeed, for
\(\zeta\in\T\), the boundary value of \(\Theta\) is given by $\Theta(\zeta)h_p=\zeta^p h_p,
\, p\geq 0.$
Therefore, $\|\Theta(\zeta)h_p\|=\|h_p\|,
\, p\geq 0,$
and hence 
\[
\Theta(\zeta)^*\Theta(\zeta)=I_{\cle},
\qquad \zeta\in\T.
\]
\end{example}

\begin{example}
Let \(\cle\) be a Hilbert space, and let \(a_1,\dots,a_m\in \D\). Define the
finite Blaschke product
\[
B(z)=\prod_{j=1}^m \frac{z-a_j}{1-\overline{a_j}z},
\qquad z\in \D.
\]
Write its Taylor expansion at the origin as $B(z)=\sum_{r=0}^{\infty} b_r z^r,
\, b_r\in \mathbb C.$
Since \(B\in H^\infty(\D)\), we have $\sum_{r=0}^{\infty}|b_r|^2=\|B\|_{H^2}^2<\infty.$ For each \(\eta\in\cle\), the vectors \(\{e_1^r\otimes\eta:r\geq0\}\) are
mutually orthogonal, and the series
\(\sum_{r=0}^{\infty} b_r e_1^r\otimes\eta\) converges in norm in
\(\clf_n^2\otimes\cle\).
Hence we define
\(L:\cle\to \clf_n^2\otimes \cle\) by
\[
L\eta=\sum_{r=0}^{\infty} b_r\, e_1^r\otimes \eta,
\qquad \eta\in \cle.
\]
Moreover, \(L\in \clb(\cle,\clf_n^2\otimes \cle)\), as $\|L\eta\|^2
=
\sum_{r=0}^{\infty}|b_r|^2\|\eta\|^2
=
\|B\|_{H^2}^2\|\eta\|^2.$
For each \(r\geq 0\), the coefficient operators \(L_r\in \clb(\cle)\) are given by
\[
L_r
=
(\langle e_1^r,\cdot\rangle\otimes I_\cle)L
=
b_r I_\cle.
\]
Hence, the associated operator-valued symbol is
\[
\Theta(z)
=
\sum_{r=0}^{\infty} z^r L_r
=
\sum_{r=0}^{\infty} b_r z^r I_\cle
=
B(z)I_\cle.
\]
Since \(B\) is a finite Blaschke product, it is inner. Therefore
\[
\Theta(\zeta)^*\Theta(\zeta)
=
|B(\zeta)|^2 I_\cle
=
I_\cle
\quad \text{for a.e. } \zeta\in\T.
\]
Thus \(\Theta\) is an inner operator-valued function. Clearly \(L\cle\subseteq\clm^\perp\). Consequently, by
Theorem~\ref{thm:operator_characterization_symbol}, the corresponding
odometer map \(W_L\) is an isometry.
This yields a family of explicit isometric odometer maps parameterized by the
zeros \(a_1,\dots,a_m\) of the finite Blaschke product.
\end{example}

\begin{remark}

\begin{enumerate}
\item In the special case $m=1$, the Blaschke product reduces to $B(z)=\frac{z-a}{1-\overline a z}, \, a\in\D.$
A direct computation of its Taylor coefficients yields
\[
b_0=-a, \qquad b_r=(1-|a|^2)\overline a^{\,r-1}, \quad r\ge1.
\]
Define $L:\cle\to \clf_n^2\otimes \cle$ by
\[
L\eta=-a\,\Omega\otimes \eta+
(1-|a|^2)\sum_{r=1}^\infty \overline a^{\,r-1}\, e_1^{r}\otimes \eta.
\]
The associated analytic symbol $\Theta(z)=\frac{z-a}{1-\overline a z}\,I_{\cle}$ is inner and hence belongs to the Schur class. The corresponding odometer map $W_L$ is an isometry.
This provides a concrete illustration of Theorem~\ref{thm:operator_characterization_symbol}, showing that the isometric nature of $W_L$ is encoded by the inner Toeplitz symbol.

\item Consider Example~7.4 of \cite{mansi2025}, where $\omega=\frac{1-\sqrt5}{2},$ and
\(L:\mathbb C\to \clf_n^2\otimes \mathbb C\) is given by
\[
L(\lambda)=\lambda (\sum_{p=0}^{\infty} c_p e_1^p)\in \clm^\perp,
\]
where $c_0=\sqrt{\frac{2}{\sqrt5+3}},$ and $
c_p=c_0\,\omega^{p-1},\, p\geq 1.$
Since $c_0=-\omega=1-\omega^2,$
a straightforward computation yields
\[
\Theta(z)
=
c_0+\sum_{p=1}^{\infty}c_0\omega^{p-1}z^p
=
-\omega+(1-\omega^2)\sum_{p=1}^{\infty}\omega^{p-1}z^p
=
\frac{z-\omega}{1-\omega z}.
\]
\end{enumerate}
\end{remark}

\begin{example}
Let $\cle=\ell^2(\N)$ with its standard orthonormal basis
$\{h_k\}_{k=1}^\infty$, and let
$S\in\clb(\cle)$ denote the unilateral shift.
Define a bounded symbol $L:\cle\to \clf_n^2\otimes \cle$
by
\[
L\eta=\sum_{r=0}^\infty \frac{1}{2^{r+1}}\, e_1^{r}\otimes S^r\eta,
\qquad \eta\in\cle.
\]
Since $S$ is an isometry, we have
\[
\|L\eta\|^2
=
\sum_{r=0}^\infty \frac{1}{4^{r+1}}
\|S^r\eta\|^2
=
\left(\sum_{r=0}^\infty \frac{1}{4^{r+1}}\right)\|\eta\|^2
=
\frac13\|\eta\|^2.
\]
Hence $L$ is bounded.
The coefficient operators are given by
\[
L_r=\frac{1}{2^{r+1}}S^r,\qquad r\ge0.
\]
Consequently, the associated operator-valued symbol is
\[
\Theta(z)
=
\sum_{r=0}^\infty z^rL_r
=
\frac12\sum_{r=0}^\infty
\left(\frac z2 S\right)^r =
\frac12\left(I_{\cle}-\frac z2 S\right)^{-1},
\]
where $\|S\|=1$ and $|z|<1$ imply the series converges in operator norm.

Thus $\Theta\in H^\infty_{\clb(\cle)}(\D)$ is a
resolvent-type operator-valued analytic function, and the Toeplitz operator
$M_{\Theta}$ provides an explicit model in which the Hardy-space shift
interacts with the unilateral shift on $\ell^2(\N)$.
However, \(\Theta\) is not inner. Indeed, for \(\zeta\in\T\), the radial
boundary value satisfies $\Theta(\zeta)h_1
=
\sum_{r=0}^{\infty}
\frac{\zeta^r}{2^{r+1}}S^r h_1
=
\sum_{r=0}^{\infty}
\frac{\zeta^r}{2^{r+1}}h_{r+1}.$
Therefore
\[
\|\Theta(\zeta)h_1\|^2
=
\sum_{r=0}^{\infty}\frac{1}{4^{r+1}}
=
\frac13
\neq 1
=
\|h_1\|^2.
\]
Hence \(\Theta(\zeta)^*\Theta(\zeta)\neq I_{\cle}\). 
This provides an example of a Toeplitz operator that is not an isometry. consequently, by
Theorem~\ref{thm:operator_characterization_symbol}, \(W_L\) is not an isometry.
\end{example}

\section{The adjoint of the odometer map} \label{sec:The adjoint of odometer map}
To analyze the spectral properties and algebraic structure associated with the odometer map, explicit knowledge of the adjoint operator $W_L^*$ is required. While \cite[Proposition~4.1]{mansi2025} established an adjoint formula in the
isometric case, we derive here, by a direct inner-product computation, an
explicit formula for the adjoint of an arbitrary bounded odometer map \(W_L\). 
This general formula provides further insight into the structure of bounded
odometer maps. We first introduce some notation that will aid in the computation of the adjoint.

Every $e_\gamma \in \clf_n^2$ can be written uniquely as $e_\gamma
=e_1^{ p}\otimes e_{\gamma'},$ where $p \geq 0$ and either $|\gamma'|=0$ or $\gamma'_1 \neq 1$. In this case, we define $\gamma(1):=p$, and for $0 \leq m \leq p,$ we define 
\begin{equation}\label{eq:e_gamma_power_m}
  e_{\gamma^{(m)}}:=e_1^{p-m} \otimes e_{\gamma'}.  
\end{equation}
Let $\{h_s\}_{s\in \Lambda}$ be an orthonormal basis for $\cle$, 
where $\Lambda$ is finite or countable. 
For each \(q\in\Lambda\), let
\[
c^{h_q}_{\mu,s}
:=
\left\langle Lh_q, e_{\mu}\otimes h_s \right\rangle,
\qquad \mu\in F_n^+,\ s\in\Lambda.
\]
With respect to the orthonormal basis 
\(\{e_\mu\otimes h_s:\mu\in F_n^+,\ s\in\Lambda\}\) of 
\(\mathcal F_n^2\otimes\mathcal{E}\), we have the Fourier expansion
\begin{equation}\label{eqn-L-hq}
Lh_q
=
\sum_{\mu \in F_n^+,\, s \in \Lambda}
c^{h_q}_{\mu,s}\, e_{\mu}\otimes h_s.
\end{equation}
We first record the action of \(L^*\) on basis vectors.

\begin{lemma} \label{lemma:adjoint of L}
Let $L: \cle \to \clf_n^2 \otimes \cle$ be a bounded linear operator. 
Then, for every \(\mu\in F_n^+\) and \(s\in\Lambda\), we have
\begin{equation*}
L^*(e_\mu \otimes h_s) = \sum_{q \in \Lambda} \ol{c^{h_q}_{\mu,s}}\, h_q.
\end{equation*}

\end{lemma}

\begin{proof}
For any $q \in \Lambda$, consider the Fourier series expansion \eqref{eqn-L-hq} of $L h_q$.
It follows that
\[
\ol{c^{h_q}_{\mu,s}}
=
\langle e_{\mu} \otimes h_s, Lh_q \rangle
=
\langle L^*(e_{\mu} \otimes h_s), h_q \rangle.
\]
Expanding $L^*(e_{\mu} \otimes h_s)$ with respect to the orthonormal basis $\{h_q\}_{q\in\Lambda}$ of $\cle$, we obtain
\[
L^*(e_{\mu} \otimes h_s)
=
\sum_{q \in \Lambda}
\langle L^*(e_{\mu} \otimes h_s), h_q \rangle h_q
=
\sum_{q \in \Lambda}
\ol{c^{h_q}_{\mu,s}}\, h_q.
\]
\end{proof}

\begin{theorem}\label{thm:general_adjoint} 
Let $W_L \in \clb(\clf_n^2 \otimes \cle)$ be the odometer map with symbol $L$. 
Then, for every \(\gamma\in F_n^+\) and \(l\in\Lambda\), we have
\begin{equation} \label{eq:general_adjoint}
\begin{split}
W_L^*(e_\gamma \otimes h_l)
&=
\chi_{\clm_0}(\gamma )\,(e_{\gamma-1} \otimes h_l)
+
\sum_{p=0}^{\gamma(1)} 
\sum_{q \in \Lambda}
\ol{c^{h_q}_{\gamma^{(p)},l}}
\, (e_n^p \otimes h_q) \\
&= \chi_{\clm_0}(\gamma )(e_{\gamma-1} \otimes h_l) +\sum^{\gamma(1)}_{p=0} e_{n}^p \otimes L^*(e_{\gamma^{(p)}}\otimes h_l),\\
    \end{split}
\end{equation}
where $\chi_{\clm_0}$ denotes the characteristic function of $\clm_0$.
\end{theorem}

\begin{proof}
With respect to the orthonormal basis $\{e_\mu \otimes h_q : \mu \in F_n^+,\, q \in \Lambda\}$
of
$\clf_n^2 \otimes \cle$, we have
\begin{align*}
W_L^*(e_\gamma \otimes h_l)
=
\sum_{\mu \in F_n^+,\, q \in \Lambda}
\langle W_L^*(e_\gamma \otimes h_l),
e_\mu \otimes h_q \rangle
\, e_\mu \otimes h_q 
=
\sum_{\mu \in F_n^+,\, q \in \Lambda}
\langle e_\gamma \otimes h_l,
W_L(e_\mu \otimes h_q) \rangle
\, e_\mu \otimes h_q.
\end{align*}
Using the decomposition 
$\clf_n^2 \otimes \cle = \cln \oplus \cln^\perp$, 
we get
\begin{equation} \label{eq: adjoint}
   W_L^*(e_\gamma \otimes h_l)
=  \sum_{\mu \in \cln_0 , q \in \Lambda} \langle  e_\gamma \otimes h_l, W_L(e_{\mu} \otimes {h_q}) \rangle e_{\mu} \otimes {h_q} + \sum_{p \in \Z_+ , q \in \Lambda} \langle e_\gamma \otimes h_l,  W_L(e_{n}^p \otimes {h_q}) \rangle e_{n}^p \otimes {h_q}.
\end{equation}
We simplify the coefficients in the first summand. Since \(\mu\in\cln_0\), it follows from \eqref{eq: e_mu+1} that
\[
\langle e_\gamma \otimes h_l, W_L(e_\mu \otimes h_q) \rangle = \langle e_\gamma \otimes h_l, e_{\mu+1} \otimes h_q \rangle = \delta_{\gamma, \mu+1} \delta_{l, q}.
\]
Notice that $\mu+1 \in \clm_0.$ Therefore
if $e_\gamma \otimes h_l \in \clm^\perp$, then the above inner product vanishes. Therefore, the first summand reduces to
\[
\sum_{\mu \in \cln_0 , q \in \Lambda} \langle  e_\gamma \otimes h_l, W_L(e_{\mu} \otimes {h_q}) \rangle e_{\mu} \otimes {h_q}
= \chi_{\clm_0}(\gamma )(e_{\gamma-1} \otimes h_l).
\] 
To simplify the second summand, we first observe that
\begin{align*}
 W_L(e_n^{p} \otimes h_q) 
= e_1^{p} \otimes Lh_q 
 = e_1^{p} \otimes \sum_{\mu \in F_n^+ , s \in \Lambda} c^{h_q}_{\mu,s} e_{\mu}\otimes h_s
= \sum_{\mu \in F_n^+ , s \in \Lambda} c^{h_q}_{\mu,s} e_1^{ p} \otimes  e_{\mu}\otimes h_s.
\end{align*}
We can write $e_{\gamma}= e_1^{\gamma(1)} \otimes e_{\gamma'}.$ If $p > \gamma(1),$ then $\langle e_\gamma \otimes h_l,  e_1^{ p} \otimes  e_{\mu}\otimes h_s \rangle =0.$ So let $0 \leq p \leq \gamma(1).$
\begin{align*}
\langle e_\gamma \otimes h_l,  W_L(e_{n}^p \otimes {h_q}) \rangle
&= \sum_{\mu \in F_n^+ , s \in \Lambda} \ol{c^{h_q}_{\mu,s} } \langle e_\gamma \otimes h_l,  e_1^{ p} \otimes  e_{\mu}\otimes h_s \rangle\\
&=\sum_{\mu \in F_n^+ } \ol{c^{h_q}_{\mu,l} } \langle e_1^{\gamma(1)} \otimes e_{\gamma'},  e_1^{ p} \otimes  e_{\mu} \rangle\\
&=\sum_{\mu \in F_n^+ } \ol{c^{h_q}_{\mu,l} } \langle e_{\gamma^{(p)}},    e_{\mu} \rangle,\qquad [\text{by }\eqref{eq:e_gamma_power_m}]\\
&= \ol{c^{h_q}_{\gamma^{(p)},l} }.    
\end{align*}
Then
\begin{align*}
 \sum_{p \in \Z_+ , q \in \Lambda} \langle e_\gamma \otimes h_l,  W_L(e_{n}^p \otimes {h_q}) \rangle e_{n}^p \otimes {h_q} 
&=  \sum_{0 \leq p \leq \gamma(1), q \in \Lambda} \langle e_\gamma \otimes h_l,  W_L(e_{n}^p \otimes {h_q}) \rangle e_{n}^p \otimes {h_q}\\
&= \sum_{0 \leq p \leq \gamma(1), q \in \Lambda}  \ol{c^{h_q}_{\gamma^{(p)},l} }\,e_{n}^p \otimes {h_q}.\\
\end{align*}
Substituting these values in~\eqref{eq: adjoint}, we get
\begin{equation*} 
   W_L^*(e_\gamma \otimes h_l)
= \chi_{\clm_0}(\gamma )(e_{\gamma-1} \otimes h_l) +\sum_{0 \leq p \leq \gamma(1), q \in \Lambda}  \ol{c^{h_q}_{\gamma^{(p)},l} }\,e_{n}^p \otimes {h_q}.
\end{equation*}
Lemma \ref{lemma:adjoint of L} implies that
\[
\sum_{0 \leq p \leq \gamma(1), q \in \Lambda}  \ol{c^{h_q}_{\gamma^{(p)},l} }\,e_{n}^p \otimes {h_q}=\sum^{\gamma(1)}_{p=0} e_{n}^p \otimes   \sum_{q \in \Lambda}  \ol{c^{h_q}_{\gamma^{(p)},l} }\,{h_q}= \sum^{\gamma(1)}_{p=0} e_n^p \otimes L^*(e_{\gamma^{(p)}}\otimes h_l).
\]
Therefore, 
\begin{equation*} 
   W_L^*(e_\gamma \otimes h_l)
= \chi_{\clm_0}(\gamma )(e_{\gamma-1} \otimes h_l) +\sum^{\gamma(1)}_{p=0}  e_{n}^p \otimes L^*(e_{\gamma^{(p)}}\otimes h_l).
\end{equation*}
This completes the proof.
\end{proof}

We now record two useful consequences of the general adjoint formula.

\begin{corollary}\label{cor:WL_adjoint}
Let $W_L \in \clb(\clf_n^2 \otimes \cle)$ be the odometer map with symbol $L$. 
\begin{enumerate}
    \item For $2 \le k \le n$, $m \ge 1$, and $\eta \in \cle$,
    \begin{equation} \label{eq:adjoint_k}
W_L^*(e_k^{ m} \otimes \eta) = e_{k-1} \otimes e_k^{ (m-1)} \otimes \eta + \Omega \otimes L^*(e_k^{m} \otimes \eta).
\end{equation}
\item For any $f  \in \clm^{\perp},$
\[
W_L^* f = \sum_{p \ge 0} e_n^p \otimes L^* (S_1^{*p} \otimes I_\cle) f.
\]
\end{enumerate}
\end{corollary}
\begin{proof}
    \begin{enumerate}
        \item Here $e_{\gamma}=e_k^{ m}$. Since $k \ge 2$, we have $\chi_{\clm_0}(\gamma) = 1$, so the first summand in~\eqref{eq:general_adjoint} becomes 
$e_{k-1} \otimes e_k^{(m-1)} \otimes \eta$. 
For the second summand, since \(\gamma(1)=0\), only the term \(p=0\) occurs.
\[
\sum_{0 \leq p \leq \gamma(1)}  e_{n}^p \otimes L^*(e_{\gamma^{(p)}}\otimes \eta)=\Omega \otimes L^*(e_k^{ m} \otimes \eta),
\]
which completes the proof.

\item 
First suppose that \(f\in\clm^\perp\) is a finite linear combination, say $f=\sum
a_{m,l}\,e_1^m\otimes h_l.$ By Theorem~\ref{thm:general_adjoint}, we have
\[
W_L^*f
=
\sum_{m,l} a_{m,l}\,W_L^*(e_1^m\otimes h_l)
=
\sum_{m,l} a_{m,l}
\sum_{p=0}^{m}
e_n^p\otimes L^*(e_1^{m-p}\otimes h_l).
\]
Reindexing these finite sums according to \(p\), we obtain
\[
W_L^*f
=
\sum_{p\geq 0}
e_n^p\otimes
\left(
\sum_{m\geq p}\sum_l
a_{m,l}\,L^*(e_1^{m-p}\otimes h_l)
\right).
\]
For any $p \ge 0$,
\[
L^*(S_1^{*p} \otimes I_\cle) f = L^*\big(\sum_{m \ge p} \sum_l a_{m,l} \, e_1^{m-p} \otimes h_l\big)
= \sum_{m \ge p} \sum_l a_{m,l} \, L^*(e_1^{m-p} \otimes h_l).
\]
Therefore
\[
W_L^* f = \sum_{p \ge 0} e_n^p \otimes L^* (S_1^{*p} \otimes I_\cle) f.
\]
Since the identity holds on a dense subspace and \(W_L^*\) is bounded, the
series on the right converges in norm to \(W_L^*f\) for every
\(f\in\clm^\perp\).
    \end{enumerate}
\end{proof}

Using Corollary~\ref{cor:WL_adjoint}, we recover the adjoint formula from
\cite[Proposition~4.1]{mansi2025} in the isometric case.
First, in view of the expansion \eqref{eqn-L-hq}, we adopt the following notation. For $p \ge 0$, we set
\[
c^{h_q}_{p,s} := \langle Lh_q,\, e_1^{ p} \otimes h_s \rangle.
\]

\begin{corollary}\label{cor:isometric_adjoint}
For 
$L \in \clb(\cle, \clf_n^2 \otimes \cle)$, 
if the odometer map $W_L$ is an isometry, then the adjoint of $W_L$ is given by

\[
W_L^* f =
\begin{cases}
\displaystyle
\sum_{p=0}^{m} \sum_{q \in \Lambda} 
\overline{c^{h_q}_{m-p,l}}\, (e_n^p \otimes h_q)=\sum^m_{p=0}  e_{n}^p \otimes L^*(e_{1}^{m-p}\otimes h_l),\quad
\text{if } f = e_1^m \otimes h_l, \\[2mm]
e_n^m \otimes e_{\mu_1-1} \otimes e_{\mu_2} \otimes \cdots \otimes e_{\mu_k} \otimes h_l,\quad 
\text{if } f = e_1^m \otimes e_{\mu_1} \otimes \cdots \otimes e_{\mu_k} \otimes h_l,\, k \geq 1, \text{ and } \mu_1 > 1,
\end{cases}
\]
for all $m \in \Z_+$ and $l \in \Lambda$.
\end{corollary}

\begin{proof}
Let $f = e_1^m \otimes h_l$. Since $(S_1^{*p} \otimes I_\cle) f=0, \, \forall p \geq m+1,$
Corollary~\ref{cor:WL_adjoint} gives 
\[
W_L^*(e_1^m \otimes h_l)
=\sum^m_{p=0}  e_{n}^p \otimes L^*(e_{1}^{m-p}\otimes h_l).
\]
Lemma \ref{lemma:adjoint of L} implies $L^*(e_{1}^{m-p}\otimes h_l) = \sum_{q \in \Lambda} \ol{c^{h_q}_{{m-p},l}}\, h_q$, and hence
\[
W_L^*(e_1^m \otimes h_l)
=\sum^m_{p=0}  e_{n}^p \otimes \sum_{q \in \Lambda} \ol{c^{h_q}_{{m-p},l}}\, h_q 
= \sum^m_{p=0} \sum_{q \in \Lambda} \ol{c^{h_q}_{{m-p},l}}\, ( e_{n}^p \otimes  h_q).
\]
Next, let $f = e_1^m \otimes e_{\mu_1}\otimes\cdots\otimes e_{\mu_k}\otimes h_l \text{ with } \mu_1>1.$
Then $W_L(e_n^m \otimes e_{\mu_1-1}\otimes\cdots\otimes e_{\mu_k} \otimes h_l)=f,$ and as $W_L$ is an isometry, the 
equality follows. 
\end{proof}

\section{Applications of the adjoint formula}\label{sec:applications_adjoint}
In this section, we use the explicit formula for \(W_L^*\) to derive several
operator-theoretic consequences for odometer maps.
We first identify the kernel of $W_L^*$ and relate it to the defect space $\cle_L\ominus L\cle$ in the isometric case. This leads to consequences for the Wold decomposition,
Fredholmness, essential normality, and a Coburn-type spectral theorem. We end with a necessary condition for hyponormality, showing that this condition is not sufficient in general.
\subsection{Kernel structure and Wold decomposition}\label{sec:Kernel}
The kernel of $W_L^*$ plays a central role in the structural analysis of the odometer map. We compute $\ker W_L^*$ explicitly and show that, in the isometric case, this adjoint kernel precisely measures the orthogonal gap between the space $\cle_L$ and the range of the symbol $L$.
Equivalently, the defect space \(\cle_L\ominus L\cle\) is the wandering
subspace for the shift part in the Wold decomposition of \(W_L\). Hence its
dimension gives the multiplicity of the shift summand and, in the
Fredholm setting, determines the Fredholm index of \(W_L\). This parallels classical Toeplitz theory, where the Fredholm index of a
Fredholm Toeplitz operator is encoded by the topological data of its symbol.

\begin{proposition}\label{thm:kernel_adjoint}
Let $W_L \in \clb(\clf_n^2 \otimes \cle)$ be an odometer map with symbol $L$. Then the kernel of its adjoint is given by
\[
\ker W_L^* = \ker W_{22}^*=\{ f \in \clm^\perp : L^* (S_1^{*p} \otimes I_\cle) f = 0 \text{ for all } p \ge 0 \} 
= \cle_L \cap (L\cle)^\perp.
\]
Moreover, if $(W_L, S^\cle)$ is an isometric representation, then 
\[
\ker W_L^* = \cle_L \ominus L\cle.
\]
\end{proposition}
\begin{proof}
From the block decomposition of $W_L$ in \eqref{eq:matrix representation}, the adjoint $W_L^*:\clm\oplus \clm^\perp \to \cln\oplus \cln^\perp$
is given by
\[
W_L^*=
\begin{pmatrix}
W_{11}^* & 0\\
W_{12}^* & W_{22}^*
\end{pmatrix}.
\]
Let $(h_1,h_2)\in \clm\oplus \clm^\perp$ satisfy $W_L^*(h_1,h_2)=0.$
Then
\[
W_{11}^*h_1=0,\qquad
W_{12}^*h_1+W_{22}^*h_2=0.
\]
Since $W_{11}:\cln\to \clm$ is unitary, its adjoint $W_{11}^*:\clm\to \cln$
is also unitary. Therefore, $h_1=0,$
which further implies $W_{22}^*h_2=0.$
Thus
\[
\ker W_L^*=\{(0,h_2): h_2\in \ker W_{22}^*\}.
\]
Identifying $\{0\}\oplus \clm^\perp$ with $\clm^\perp$, we obtain
\[
\ker W_L^*=\ker W_{22}^*\subseteq \clm^\perp.
\]
Now let $f\in \ker W_L^*$. Since $\ker W_L^*\subseteq \clm^\perp$, Corollary~\ref{cor:WL_adjoint} yields
\[
0=W_L^*f=\sum_{p\ge 0} e_n^p\otimes L^*(S_1^{*p}\otimes I_\cle)f.
\]
As the subspaces $e_n^p\otimes \cle$, $p\ge 0$, are mutually orthogonal, it follows that
\[
L^*(S_1^{*p}\otimes I_\cle)f=0
\qquad \text{for all } p\ge 0.
\]
Hence
\[
\ker W_L^*
\subseteq
\{ f \in \clm^\perp : L^*(S_1^{*p}\otimes I_\cle)f=0
\text{ for all } p\ge 0\}.
\]
Conversely, let $f \in \{ g \in \clm^\perp : L^*(S_1^{*p}\otimes I_\cle)g=0
\text{ for all } p\ge 0\}.$
Since $f\in \clm^\perp$, Corollary~\ref{cor:WL_adjoint} implies $W_L^*f=
\sum_{p\ge 0} e_n^p\otimes L^*(S_1^{*p}\otimes I_\cle)f=0.$
Therefore
\[
\ker W_L^*
=
\{ f \in \clm^\perp : L^*(S_1^{*p}\otimes I_\cle)f=0
\text{ for all } p\ge 0\},
\]
proving the required equality. This is equivalent to 
\[
\langle L^*(S_1^{*p}\otimes I_\cle)f,\zeta\rangle_\cle=
\langle f,(S_1^p\otimes I_\cle)L\zeta\rangle_{\clf_n^2\otimes\cle}=0, \quad \forall\,\zeta \in \cle,\, p\ge0.
\]
For $p=0$, this gives $f\perp L\cle$. Therefore, $f \in (L\cle)^\perp.$
For each $p \ge 1$, it follows that $f \perp (S_1^{p}\otimes I_\cle)L\cle.$
Moreover, since $f \in \clm^\perp$, the definition of $\cle_L$ in \eqref{eq:EL} implies
\[
\ker W_L^* \subseteq \cle_L \cap (L\cle)^\perp.
\]

Conversely, if $f \in \cle_L \cap (L\cle)^\perp$, then by definition $f \perp (S_1^{p}\otimes I_\cle)L\cle$
for all $p \ge 0.$ Therefore, Corollary~\ref{cor:WL_adjoint} implies $ W_L^*f=0,$ and hence
\[
\ker W_L^* = \cle_L \cap (L\cle)^\perp.
\]

Finally, suppose $W_L$ is an isometry. Theorem~\ref{thm:fock_classification} implies $L$ is an isometry and $L\cle \subseteq \cle_L.$ Therefore, $L\cle$ is a closed subspace of $\cle_L$. Consequently, 
\[
\ker W_L^* =\cle_L \ominus L\cle.
\] 
\end{proof}

The following example illustrates the preceding proposition by showing that, for an isometric odometer map, the multiplicity of the shift part of $W_L$ is determined by the dimension of the defect space $\cle_L\ominus L\cle$.

\begin{example}\label{rem:wold_and_mult}
Let
\(L:\cle\to \clf_n^2\otimes\cle\) be given by
\[
L\eta=e_1\otimes \eta,\qquad \eta\in\cle.
\]
Then, for \(p\geq 0\), $W_L(e_n^p\otimes \eta)
=
e_1^p\otimes L\eta
=
e_1^{p+1}\otimes \eta.$
Hence the range of \(W_L\) in \(\clm^\perp\) is $\overline{\operatorname{span}}
\{e_1^m\otimes \eta:m\geq 1,\ \eta\in\cle\}.$
Then \(W_L\) is an isometry, and
\[
\ker W_L^*
=
(\operatorname{Ran}W_L)^\perp
=
\Omega\otimes\cle.
\]
For \(p\geq 1\) and \(\eta\in\cle\), we have $(S_1^p\otimes I_\cle)L\eta
=
e_1^{p+1}\otimes\eta.$
Hence $\overline{\operatorname{span}}
\{(S_1^p\otimes I_\cle)L\eta:p\geq 1,\ \eta\in\cle\}
=
\bigoplus_{j\geq 2} e_1^j\otimes\cle.$
Therefore
\[
\cle_L
=
\clm^\perp
\ominus
\overline{\operatorname{span}}
\{(S_1^p\otimes I_\cle)L\eta:p\geq 1,\ \eta\in\cle\}
=
(\Omega\otimes\cle)\oplus(e_1\otimes\cle).
\]
Consequently,
\[
\cle_L\ominus L\cle
=
\big((\Omega\otimes\cle)\oplus(e_1\otimes\cle)\big)
\ominus
(e_1\otimes\cle)
=
\Omega\otimes\cle.
\]
Hence, 
\[
\ker W_L^*=\cle_L\ominus L\cle, \text{ and } \dim\ker W_L^*
=
\dim\cle.
\]
Furthermore, using \eqref{eq:defn-W_L}, one checks that $\clf_n^2\otimes\cle
=
\bigoplus_{m\geq 0} W_L^m(\ker W_L^*).$
Therefore, $W_L\cong M_z$ on $H^2_\cle(\D)$.
Consequently, $\sigma(W_L) = \sigma(M_z) = \ol{\D}.$

More generally, fix \(k\geq 1\), and let
\(L:\cle\to \clf_n^2\otimes\cle\) be given by $L\eta=e_1^k\otimes \eta,\, \eta\in\cle.$
Then for $p\geq 1$, $\overline{\operatorname{span}}
\{(S_1^p\otimes I_\cle)L\eta:p\geq 1,\eta\in\cle\}
=
\bigoplus_{j\geq k+1} e_1^j\otimes \cle,$ and hence
\[
\cle_L
=
\clm^\perp
\ominus
\overline{\operatorname{span}}
\{(S_1^p\otimes I_\cle)L\eta:p\geq 1,\eta\in\cle\}
=
\bigoplus_{j=0}^{k} e_1^j\otimes \cle.
\]
Consequently, 
$$\ker W_L^*=\cle_L\ominus L\cle=
\bigoplus_{j=0}^{k-1} e_1^j\otimes \cle,\, \text{ and } \dim\ker W_L^*
=
k\dim\cle.$$  
Again, using \eqref{eq:defn-W_L}, one checks that \(W_L\) is a pure isometry
and is unitarily equivalent to the unilateral shift of multiplicity
\(k\dim\cle\).
\end{example}

We next obtain another characterization of the unitarity of \(W_L\), which recovers \cite[Theorem~4.4]{mansi2025} as a direct consequence of the preceding proposition.

\begin{corollary}\label{cor:unitary_WL}
Let $(W_L,S^\cle)$ be an isometric Fock representation of $\clo_n$. Then the
odometer map $W_L$ is unitary if and only if $L\cle=\Omega\otimes\cle.$
\end{corollary}
\begin{proof}
Assume first that \(W_L\) is unitary. Then $\ker W_L^*=\{0\}.$
By Proposition~\ref{thm:kernel_adjoint}, this is equivalent to $\cle_L=L\cle.$ By the definition of
$\cle_L$, we obtain the orthogonal sum 
\[
\clm^\perp
=
L\cle
\oplus
\overline{\operatorname{span}}
\{(S_1^p\otimes I_\cle)L\zeta:\ p\ge1,\ \zeta\in\cle\}
=
\bigoplus_{p\ge0}(S_1^p\otimes I_\cle)L\cle,
\]
as for $q>p\ge0$,  $L\cle=\cle_L \perp \overline{\operatorname{span}}
\{(S_1^r\otimes I_\cle)L\eta:\ r\ge1,\ \eta\in\cle\}$ implies
\[
\left\langle
(S_1^p\otimes I_\cle)L\zeta,
(S_1^q\otimes I_\cle)L\eta
\right\rangle
=
\left\langle
L\zeta,
(S_1^{q-p}\otimes I_\cle)L\eta
\right\rangle
=0.
\]
Thus $L\cle$ is a generating wandering subspace for
$\restr{S_1\otimes I_\cle}{\clm^\perp}$. On the other hand, $\restr{S_1\otimes I_\cle}{\clm^\perp}$
is the unilateral shift on $\clm^\perp,$
whose wandering subspace is precisely
\[
\ker\left(\restr{(S_1\otimes I_\cle)^*}{\clm^\perp}\right)
=
\Omega\otimes\cle.
\]
By uniqueness of the wandering subspace in the Wold decomposition, we obtain
\[
L\cle=\Omega\otimes\cle.
\]

Conversely, suppose that $L\cle=\Omega\otimes\cle$. Then $\overline{\operatorname{span}}
\{(S_1^p\otimes I_\cle)L\cle:\ p\ge1\}
=
\overline{\operatorname{span}}\{e_1^p\otimes\cle:\ p\ge1\}.$
Hence, by the definition of $\cle_L$,
\[
\cle_L
=
\clm^\perp
\ominus
\overline{\operatorname{span}}\{e_1^p\otimes\cle:\ p\ge1\}
=
\Omega\otimes\cle
=
L\cle.
\]
Thus $\ker W_L^*
=
\cle_L\ominus L\cle
=
\{0\}.$
Since $W_L$ is an isometry, it follows that $W_L$ is unitary.
\end{proof}

By Proposition~\ref{thm:kernel_adjoint}, the shift part in the Wold decomposition of an isometric map \(W_L\) on
\(\clf_n^2\otimes\cle\) is given by
\[
\clh_s
=
\bigoplus_{k=0}^{\infty} W_L^k(\ker W_L^*)
=
\bigoplus_{k=0}^{\infty} W_L^k(\cle_L\ominus L\cle).
\]
Moreover, since \(W_L\) is an isometry, Theorem~\ref{thm:operator_characterization_symbol}
implies that the associated multiplication operator
\(M_\Theta\in\clb(H^2_\cle(\D))\) is also an isometry. Its corresponding
pure shift part is therefore
\[
\clh_s^\Theta
=
\bigoplus_{k=0}^{\infty} M_\Theta^k(\ker M_\Theta^*).
\]
The following result shows that the multiplicity of the shift part of \(W_L\)
is encoded by the associated analytic Toeplitz operator \(M_\Theta\).

\begin{theorem}\label{thm:wold_decomposition}
Let $(W_L, S^\cle)$ be an isometric representation. Then 
\[
\ker M_\Theta^*=U_1(\cle_L \ominus L\cle),\, \text{and }
\operatorname{mult}(W_L)=\operatorname{mult}(M_\Theta).
\]
\end{theorem}

\begin{proof}
Theorem~\ref{prop:hardy_space_multiplier} implies that $M_\Theta = U_1 W_{22} U_n^*$, and hence $M_\Theta^* = U_n W_{22}^* U_1^*$. Let $f \in H^2_\cle(\D)$. Then $f \in \ker M_\Theta^*$ if and only if
\[
U_n W_{22}^* U_1^* f = 0 \iff W_{22}^* (U_1^* f) = 0 \iff U_1^* f \in \ker W_{22}^*.
\]
By Proposition~\ref{thm:kernel_adjoint}, we know $\ker W_{22}^* = \ker W_L^* = \cle_L \ominus L\cle$. Since $\cle_L \subseteq \clm^\perp$, and $U_1 : \clm^\perp \to H^2_\cle(\D)$ is a unitary operator, we obtain
\[
\ker M_\Theta^* = U_1(\ker W_{22}^*) = U_1(\cle_L \ominus L\cle).
\]
Therefore,
\[
\operatorname{mult}(M_\Theta) = \dim(\ker M_\Theta^*) = \dim \big(U_1(\cle_L \ominus L\cle)\big) = \dim(\cle_L \ominus L\cle) = \dim(\ker W_L^*) = \operatorname{mult}(W_L).
\]
\end{proof}

\subsection{Fredholm theory }\label{sec:FTCT}
Corollary~\ref{cor:unitary_WL} corresponds to the vanishing of the defect space. More generally, finite-dimensional defect characterizes the Fredholm and essential normality properties of $W_L$.

\begin{proposition}\label{prop:fredholm_essentially_normal}
Let $(W_L, S^\cle)$ be an isometric representation. Then the following are equivalent:
\begin{enumerate}
    \item $W_L$ is Fredholm;
    \item $W_L$ is essentially normal;
    \item $\dim(\cle_L\ominus L\cle)<\infty$.
\end{enumerate}
In this case,
\[
\ind(W_L)=-\,\dim(\cle_L\ominus L\cle).
\]
\end{proposition}

\begin{proof}
We first show that $(2)$ and $(3)$ are equivalent. Since $W_L$ is an isometry,
\[
[W_L^*,W_L]
=
W_L^*W_L-W_LW_L^*
=
I-W_LW_L^*
=P_{\ker W_L^*}.
\]
It follows that $W_L$ is essentially normal if and only if $P_{\ker W_L^*}$ is compact. Since an orthogonal projection is compact if and only if its range is finite-dimensional, using Proposition~\ref{thm:kernel_adjoint}, we conclude that
\[
W_L \text{ is essentially normal }
\iff
\dim\ker W_L^*=\dim(\cle_L\ominus L\cle)<\infty.
\]
This proves the equivalence of $(2)$ and $(3)$.

Next, since $W_L$ is an isometry, we have $\ker W_L=\{0\}$ and $\ran W_L$ is closed. Thus $W_L$ is Fredholm if and only if $\dim(\ran W_L)^\perp
=\dim\ker W_L^*
<\infty.$ Therefore
\[
W_L \text{ is Fredholm }
\iff
\dim(\cle_L\ominus L\cle)<\infty.
\]
This proves the equivalence of $(1)$ and $(3)$.

Finally, if these equivalent conditions hold, then the Fredholm index of $W_L$ is
\[
\ind(W_L)
=
\dim\ker W_L-\dim\ker W_L^*=
-\dim\ker W_L^*
=-\,\dim(\cle_L\ominus L\cle).
\]
\end{proof}

The following example illustrates that the Fredholmness of the odometer map is completely determined by the dimension of the defect space \(\cle_L\ominus L\cle\).

\begin{example}\label{ex:fredholm_shift_case}
Consider the symbol $L$ from Example~\ref{rem:wold_and_mult}, given by $L\eta = e_1 \otimes \eta, \, \eta \in \cle.$
As computed there, $\ker W_L^*
=
\Omega\otimes\cle.$ Hence $W_L$ is Fredholm if and only if $\dim \cle < \infty$. But  $\dim \cle=\dim(\cle_L \ominus L\cle).$ Thus this example verifies Proposition~\ref{prop:fredholm_essentially_normal}. Moreover,
\[
\ind(W_L)
=
-\dim(\ker W_L^*)
=
-\dim\cle.
\]

In particular, if $\cle=\C^m$ for some $m\in\mathbb N$, then $W_L$ is Fredholm
and $\ind(W_L)=-m,$
the same index as the classical vector-valued unilateral shift of multiplicity
\(m\) on \(H^2_{\mathbb C^m}(\D)\).
\end{example}

The next example exhibits a nontrivial family of isometric odometer maps for which the defect space, Fredholmness, and Fredholm index can be computed explicitly.

\begin{example}\label{ex:projection_fredholm}
Let $P\in\clb(\cle)$ be an orthogonal projection, and define
\[
L:\cle\to \clf_n^2\otimes \cle,
\qquad
L\eta=\Omega\otimes (I-P)\eta+e_1\otimes P\eta,
\qquad \eta\in\cle.
\]
Then $L$ is an isometry, since for every $\eta\in\cle$,
\[
\|L\eta\|^2
=
\|(I-P)\eta\|^2+\|P\eta\|^2
=
\|\eta\|^2.
\]
We now compute $\cle_L$. 
Since $e_1^{ p}\otimes L\zeta
=
e_1^{p}\otimes (I-P)\zeta
+
e_1^{(p+1)}\otimes P\zeta,$
it follows that
\[
\overline{\operatorname{span}}\{e_1^{p}\otimes L\zeta:\ p\ge1,\ \zeta\in\cle\}
=
\left(\bigoplus_{m\ge1} e_1^{ m}\otimes (I-P)\cle\right)
\oplus
\left(\bigoplus_{m\ge2} e_1^{ m}\otimes P\cle\right).
\]
Therefore, by \eqref{eq:EL}, $\cle_L=(\Omega\otimes \cle)\oplus (e_1\otimes P\cle),$ and by Theorem~\ref{thm:fock_classification}, $(W_L,S^\cle)$ is an isometric representation.
On the other hand, $L\cle=(\Omega\otimes (I-P)\cle)\oplus (e_1\otimes P\cle).$
Hence $\cle_L\ominus L\cle=\Omega\otimes P\cle.$
Consequently,
\[
\dim(\cle_L\ominus L\cle)=\dim(P\cle)=\rank P.
\]
It now follows from Proposition~\ref{prop:fredholm_essentially_normal} that $W_L$ is Fredholm if and only if $\rank P<\infty$, and in that case
\[
\ind(W_L)=-\,\rank P.
\]
\end{example}

\subsection{Spectral consequences and a Coburn-type theorem}\label{sec:coburn_type}
A classical consequence of Coburn's theorem states that if $T$ is a proper isometric Toeplitz operator of finite multiplicity, then its essential spectrum is the unit circle. Moreover, for every \(\lambda\in\D\), one has  (cf. \cite{coburn1966}):
\[
\ind(T-\lambda I) = -\dim\ker T^*, \qquad \text{for all } \lambda \in \D.
\]
Inspired by this classical result, we use our earlier characterization of Fredholmness to establish a similar spectral theorem for odometer maps on the free Fock space.

\begin{theorem}\label{thm:coburn_odometer}
Let \(W_L\) be an isometric odometer map and suppose that $\dim(\cle_L\ominus L\cle)<\infty.$
Then, for every \(\lambda\in\D\), the operator \(W_L-\lambda I\) is Fredholm and
\[
\ind(W_L-\lambda I)
=
-\dim\ker W_L^*
=
-\dim(\cle_L\ominus L\cle).
\]
If, in addition, \(W_L\) is not unitary, then $\sigma_{\mathrm{ess}}(W_L)=\T.$
\end{theorem}
\begin{proof}
As $\dim(\cle_L\ominus L\cle)<\infty,$ 
Proposition~\ref{prop:fredholm_essentially_normal} yields that $W_L$ is Fredholm with $\ind(W_L)
=
-\dim(\cle_L\ominus L\cle).$
Let $\lambda \in \D$. Since $W_L$ is an isometry, the reverse triangle inequality implies
\[
\|(W_L-\lambda I)x\|
\ge
\|W_Lx\|-|\lambda|\|x\|
=
(1-|\lambda|)\|x\|.
\]
Since $|\lambda|<1$, it follows that $W_L-\lambda I$ is bounded below. In particular, its kernel is trivial and its range is closed. Therefore, $W_L-\lambda I$ is semi-Fredholm for every $\lambda \in \D$. Now the map
\[
\lambda\mapsto W_L-\lambda I
\]
is continuous in the operator norm, and its image lies in the set of semi-Fredholm operators. Since the index is locally constant on the semi-Fredholm set and $\D$ is connected, we conclude that
\[
\ind(W_L-\lambda I)=\ind(W_L)=
-\dim(\cle_L\ominus L\cle)
\qquad (\lambda\in\D).
\]
Since $\ker(W_L-\lambda I)=\{0\}$ and the index is finite, it follows that the cokernel is finite-dimensional. Hence $W_L-\lambda I$ is Fredholm for every $\lambda\in\D$.  This proves the first assertion.

Consequently $\D\cap\sigma_{\mathrm{ess}}(W_L)=\varnothing.$
Since $W_L$ is an isometry, $\sigma(W_L)\subseteq \overline{\D}.$
Consequently, $\sigma_{\mathrm{ess}}(W_L)\subseteq \T.$
Now assume, in addition, that \(W_L\) is not unitary. Then \(W_L\) is a proper
isometry. By the Wold decomposition, $W_L \cong U\oplus S_d,$
where \(U\) is unitary and \(S_d\) is the unilateral shift of multiplicity
\[
0<d=\dim\ker W_L^*
=
\dim(\cle_L\ominus L\cle)<\infty.
\]
The essential spectrum of the unilateral shift of finite positive multiplicity
is \(\T\). Therefore,
\[
\T
=
\sigma_{\mathrm{ess}}(S_d)
\subseteq
\sigma_{\mathrm{ess}}(U\oplus S_d)
=
\sigma_{\mathrm{ess}}(W_L).
\]
Combining both inclusions
we obtain
\[
\sigma_{\mathrm{ess}}(W_L)=\T.
\]
\end{proof}

\subsection{Hyponormality} \label{sec:hyponormal}
\begin{theorem}\label{thm:hyponormality_bound}
Assume $n \ge 2$. Let $W_L \in \clb(\clf_n^2 \otimes \cle)$ be an odometer map with symbol $L$. If $W_L$ is a hyponormal operator, then 
\[
\|L\eta\| \ge \|\eta\| \quad \text{for all } \eta \in \cle.
\]
Consequently, if \(\|L\|<1\) or if \(L\) has a non-trivial kernel, then
\(W_L\) cannot be hyponormal and, consequently, cannot be subnormal.
\end{theorem}

\begin{proof}
Let $W_L$ be hyponormal. This requires $\|W_L x\| \ge \|W_L^* x\|$ for all $x \in \clf_n^2 \otimes \cle$. Fix $\eta \in \cle$ and consider $x = e_n \otimes \eta$. 
By the definition of the odometer map, $W_L(e_n \otimes \eta) = e_1 \otimes L\eta$, and hence 
\[
\|W_L x\| = \|e_1 \otimes L\eta\| = \|L\eta\|.
\]
To evaluate the adjoint action on $x$, we apply the explicit formula from Corollary~\ref{cor:WL_adjoint}(1) with $k=n$ and $m=1$:
\[
W_L^*(e_n \otimes \eta) = e_{n-1} \otimes \eta + \Omega \otimes L^*(e_n \otimes \eta).
\]
Since \(n\ge2\), the vector \(e_{n-1}\) is a word of length one, and hence
\(e_{n-1}\otimes\eta\perp \Omega\otimes\cle\). Therefore,
\[
\|W_L^* x\|^2 = \|e_{n-1} \otimes \eta\|^2 + \|\Omega \otimes L^*(e_n \otimes \eta)\|^2 = \|\eta\|^2 + \|L^*(e_n \otimes \eta)\|^2.
\]
This implies $\|W_L^* x\|^2 \ge \|\eta\|^2$. Enforcing the hyponormality condition $\|W_L x\| \ge \|W_L^* x\|$ yields $\|L\eta\| \ge \|\eta\|$, completing the proof.
\end{proof}

We now show that the lower-bound condition in Theorem~\ref{thm:hyponormality_bound}
is necessary but not sufficient for hyponormality.

\begin{example}
Assume $n\ge2$, and fix a unit vector $h\in\cle$. Define
\(L:\cle\to\clf_n^2\otimes\cle\) by
\[
L\eta=\Omega\otimes\eta+e_1\otimes\eta,
\qquad \eta\in\cle.
\]
Then
\[
\|L\eta\|=\sqrt2\,\|\eta\|\geq \|\eta\|,
\qquad \eta\in\cle.
\]
Thus $L$ satisfies the necessary lower bound condition.
We show that $W_L$ is not hyponormal. Let $x=e_1\otimes h$. Then
\[
W_Lx=e_2\otimes h,
\qquad
\|W_Lx\|=1.
\]
On the other hand, the general adjoint formula gives
\[
W_L^*x=\Omega\otimes h+e_n\otimes h.
\]
Hence
\[
\|W_L^*x\|^2=2>1=\|W_Lx\|^2.
\]
Therefore, $W_L$ is not hyponormal.
\end{example}

\noindent\textbf{Acknowledgements.}
The author is deeply grateful to Professor Baruch Solel for carefully reading an
earlier draft of this manuscript and for his valuable comments and suggestions,
which helped improve the presentation and clarity of the paper. The author also
thanks the Department of Mathematics, Technion--Israel Institute of Technology,
Haifa, Israel, for its institutional support and for providing a conducive
research environment during the preparation of this work.

\end{document}